\documentclass{amsart}

\usepackage{amsmath, amssymb, amsthm}

\theoremstyle{plain}
\usepackage[leqno]{amsmath}
\usepackage[numbers,sort&compress]{natbib}
\usepackage[mathscr]{euscript}
\usepackage{graphicx} 
\usepackage{booktabs}
\usepackage{color}
\usepackage{framed}
\usepackage{hyperref}

\newtheorem{algorithm}{Algorithm}[section]
\newtheorem {theorem}{Theorem}[section]

\newtheorem {lemma}{Lemma}[section]

\newtheorem {definition}{Definition}[section]
\newtheorem {remark}{Remark}[section]
\newtheorem {proposition}{Proposition}[section]

\newenvironment{dedication}

\newcommand{\co}{{\rm conv\,}}

\newcommand{\dom}{{\rm dom\,}}

\numberwithin{equation}{section}

\title[Convergence analysis of proximal-type algorithms for DC programs]{Convergence analysis of a proximal-type algorithm for DC programs with applications to variable selection}

\author[Shuang Wu]{Shuang Wu}
\address[Shuang Wu]{Key Laboratory for Applied Statistics of MOE, School of Mathematics and Statistics, Northeast Normal University, Changchun 130024, China}
\email{\tt wus687@nenu.edu.cn}

\author[Bui Van Dinh]{Bui Van Dinh}
\address[Bui Van Dinh]{Department of Mathematics, Le Quy Don Technical University, Hanoi, Vietnam}
\email{\tt vandinhb@gmail.com; vandinhb@lqdtu.edu.vn}

\author[Liguo Jiao]{Liguo Jiao}
\address[Liguo Jiao]{Academy for Advanced Interdisciplinary Studies, Northeast Normal University, Changchun 130024, China}
\email{\tt jiaolg356@nenu.edu.cn; hanchezi@163.com}

\author[Do Sang Kim]{Do Sang Kim$^{\ast}$}
\address[Do Sang Kim]{Department of Applied Mathematics, Pukyong National University, Busan 48513, Korea}
\email{\tt dskim@pknu.ac.kr}

\author[Wensheng Zhu]{Wensheng Zhu$^{\ast}$}
\address[Wensheng Zhu]{Key Laboratory for Applied Statistics of MOE, School of Mathematics and Statistics, Northeast Normal University, Changchun 130024, China}
\email{\tt wszhu@nenu.edu.cn}

\thanks{$^{\ast}$Corresponding Authors}



\begin{document}

\maketitle
\begin{dedication}
\begin{center}
{\it\small Dedicated to Professor Tamaki Tanaka on the occasion of his 65th birthday with respect}
\end{center}
\end{dedication}

\begin{abstract} 
We consider a minimization problem of the form $P(\varphi, g, h):$
$$\min\left\{f(x):= \varphi(x) + g(x) - h(x) \colon x \in \mathbb{R}^n\right\},$$ 
where $\varphi$ is a differentiable function and $g,$ $h$ are convex functions, and introduce iterative methods to finding a critical point of $f$ when $f$ is differentiable. 
We show that the point computed by proximal point algorithm at each iteration can be used to determine a descent direction for the objective function at this point. 
This algorithm can be considered as a combination of proximal point algorithm together with a linesearch step that uses this descent direction.
We also study convergence results of these algorithms and the inertial proximal methods proposed by Maing$\acute{e}$ and Moudafi (SIAM J. Optim. {\bf 19}(2008), 397--413) under the main assumption that the objective function satisfies the Kurdika--{\L}ojasiewicz property.
The proposed algrithm is then applied to solve the variable selection problem in linear regression.
\end{abstract}


\section{Introduction}\label{Sec:1}
Let $\Bbb R^n$ be an $n$-dimensional Euclidean space equipped with the inner product $\langle\cdot, \cdot\rangle$ and the associated norm $\|\cdot\|$. Let $f:\Bbb R^n\rightarrow\Bbb R\cup\{+\infty\}$ be a nonconvex function such that $f$ can be decomposed in the form
               $$f(x)=\varphi(x)+g(x)-h(x),$$
 where $\varphi:\Bbb R^n\rightarrow\Bbb R$ is a continuously differentiable function (not necessarily convex) and $g, h:\Bbb R^n\rightarrow\Bbb R\cup\{+\infty\}$ are convex, proper lower semicontinuous functions. We consider the following optimization problem
\begin{align}\label{problem}
\min_{x\in\Bbb R^n}\ \left\{f(x):=\varphi(x)+g(x)-h(x) \right\}.\tag{$P(\varphi, g, h)$}
\end{align}
The problem in this form has been investigated in some recent papers, such as Maing\'{e} and Moudafi \cite{MM} and An and Nam \cite{AN}. The special structure of this problem allows us the use of the powerful tools in convex analysis and convex optimization. The differentiability of $\varphi$ and the convexity of $g$ and $h$ of the objective function would be employed to develop appropriate tools from both theoretical and algorithmic point of views.

When $\varphi$ is convex (or $\varphi\equiv 0$), the function $f$ is called a {\it DC function} ({\bf D}ifference of two {\bf C}onvex function). It is worth mentioning that the class of DC functions contains all lower  $\mathcal{C}^2$ function and is closed under all operations usually considered in optimization. 
Some theoretical aspects, such as optimality conditions and duality theorems, related to DC programs are studied in the literature; see, e.g., \cite{FK, HK, SK}.

For solving DC programs from the convex analysis approach, DC algorithms (DCA), based on local optimality conditions and duality in DC programming, have been introduced by T. Pham Dinh \cite{PS} in 1986 as an extension of the subgradient algorithms to DC programming and extensively developed by H.A. Le Thi and T. Pham Dinh \cite{TA2} since 1994. Since then many authors have contributed to providing mathematical foundation for the algorithm and making it accessible for application, see \cite{AT1, AT2, ATM, ANT, DC, TA1} and references quoted therein.

One of the most important methods to handle ill-posed problems is the celebrated proximal point method. 
This method was firstly introduced by Martinet \cite{MAR} in 1970 for solving convex minimization problems and then extensively developed by Rockafellar \cite{ROC} to finding a zero point of a maximal monotone inclusion problem; see \cite{Beck2017,ParikhBoyd2014} for more detailed information. 
Since then, many researchers have succeeded in applying this method for solving many other problems, such as variational inequality problems \cite{FP} and equilibrium problems \cite{Mou}. 
Moreover, Sun {\it et al.} \cite{SSC} (see also \cite{ABB, BSM, MM2, SONS} for more proximal point algorithms) and An {\it et al.} \cite{AN} applied the proximal point method to DC optimization and problem \eqref{problem}, respectively.

It is obvious that, with a suitable DC decomposition, DCA becomes a proximal point algorithm. Very recently, Artacho {\it et al.} \cite{AFV} introduced the so-called {\it boosted DC algorithm} with backtracking for solving differentiable DC programming. This method can be considered as a combination of DCA and the descent algorithm proposed by Fukushima and Mine \cite{MMH,MH} to force the value of the objective function at each iteration reduces more than that performed by DCA.

Along with proximal point algorithm \cite{AN}, the inertial proximal method was proposed for solving \eqref{problem} by Maing\'{e} and Moudafi \cite{MM} (see also \cite{Alv, AA} for more inertial proximal methods). Although this algorithm has been known since 2008, its convergence analysis for general classes of difference functions, to our best knowledge, is still an open research question.

In this paper, we first introduce an algorithm called {\it boosted proximal point algorithm} to finding a stationary point of \eqref{problem} when $f$ is differentiable. 
This algorithm can be seen as a combination of the proximal point method and the descent algorithm proposed by Fukushima and Mine \cite{MMH} to make the value of the objective function $f$ at each iteration reduce much more than that in  the proximal point algorithm. The global convergence of the proposed algorithm and its convergence rate are obtained under the main assumption that the objective function satisfies {\L}ojasiewicz inequality \cite{LS}. 
Based on the method developed recently in \cite{BAA}, we then prove the global convergence of the inertial proximal algorithm for \eqref{problem} provided that $f$ posses the Kurdyka--{\L}ojasiewicz property.

The rest of paper is organized as follows. In Section \ref{Sec:2}, some tools of variational analysis are recalled. The boosted proximal point algorithm for \eqref{problem} and its convergent analysis are presented in Section \ref{Sec:3}. Section \ref{Sec:4} is devoted to presenting the convergence of the inertial proximal algorithm for \eqref{problem}. 
A numerical example is given in Section \ref{Sec:5}.
An application to solve the variable selection problem in linear regression is studied by using the proposed algorithm in Section \ref{Sec:6}.
Conclusions are given in Section \ref{Sec:7}.

\section{Preliminaries}\label{Sec:2}
This section contains the necessary preliminaries needed throughout the paper. We start with generalized differentiation for nonsmooth functions referring the reader to the monographs \cite{CF1, MB1, RRT} for more details and commentaries.

Let us denote the nonnegative orthant in $\Bbb R^n$  by $\Bbb R_{+}^n = [0, \infty)^n$ and $\Bbb B(x, r)$ the closed ball of center $x$ and radius $r > 0$. The gradient of a differentiable function $f: \Bbb R^n \rightarrow \Bbb R^m$ at some point $x\in \Bbb R^n$ is denoted by $\nabla f(x) \in \Bbb R^{n \times m}$.

For an extended-real-value function $f :\Bbb R^n \rightarrow \bar{\Bbb R} := \Bbb R \cup \{ +\infty \}$, the domain of $f$ is the set $\dom f = \{x \in \Bbb R^n \colon f(x) < + \infty\}$, moreover $f$ is said to be {\it proper} if its domain is nonempty, and $f$ is said to be {\it coercive} if $f(x) \rightarrow + \infty$, whenever $\| x\| \rightarrow +\infty$.

Let $\Omega\subset \Bbb R^n$ and $\Omega\not= \emptyset$, we use the notation $d(\bar x; \Omega)$ to denote the distance from $\bar x$ to $\Omega$, i.e.,
$$d(\bar x; \Omega)=\inf\limits_{x\in\Omega}\| x-\bar x\|.$$

Recall that a function $f:\Bbb R^n\rightarrow \Bbb{\bar R}$ is said to be {\it strongly convex} with $\tau>0$ if
$$f\left(\lambda x+(1-\lambda)y\right)\leq \lambda f(x)+(1-\lambda)f(y) -\frac{\tau}{2}\lambda (1-\lambda)\|x-y\|^2,$$
for all $x,y\in\Bbb R^n$ and $\lambda\in (0,1)$. Moreover, if $\tau=0$, $f$ is said to be {\it convex}. Clearly, $f$ is strongly convex if and only if $f-\frac{\tau}{2}\|\cdot\|^2$ is convex; see, e.g., \cite[Theorem 5.17]{Beck2017}.

The function $f:\Bbb R^n\rightarrow \Bbb R$ is called {\it Lipschitz continuous} if there exists a positive constant $L$ such that
$$\|f(x)-f(y)\|\leq L\|x-y\|, \hbox{ for all $x,y\in\Bbb R^n$.}$$
Further, $f$ is called {\it locally Lipschitz continuous} if for every $x\in\Bbb R^n$, there exists a neighborhood $V$ of $x$ such that $f$ restricted to $V$ is Lipschitz continuous.

Given a lower semicontinuous function $f : \Bbb R^n \rightarrow \bar{\Bbb R}$, we use the symbol $z \xrightarrow{f} x$ to indicate that $z \rightarrow x$ and $f(z) \rightarrow f(x)$. The {\it Fr\'{e}chet subdifferential} of $f$ at $\bar x \in  dom f$ is defined by
$$\partial^Ff(\bar x)=\left\{v\in \Bbb R^n \colon \liminf\limits_{x\rightarrow\bar x}\frac{f(x)-f(\bar x)-\langle v, x-\bar x\rangle}{\| x-\bar x\|}\geq 0\right\}.$$
Set $\partial^Ff(\bar x):=\emptyset$ if $\bar x\not\in \dom f$. 
It is worth noting that the Fr\'{e}chet subdifferential mapping does not have a closed graph, and it is unstable computationally. Based on the Fr\'{e}chet subdifferential, the {\it limiting subdifferential} of $f$ at $\bar x\in dom f$ (known also as the {\it basic}, or {\it Mordukhovich subdifferential}) is defined by
$$\partial^Mf(\bar x)=\limsup\limits_{x\xrightarrow{f} \bar x}\partial^Ff(x):=\left\{v\in \Bbb R^n \colon \exists x^k\xrightarrow{f}\bar x, \ v^k\in \partial^Ff(x^k),\ v^k\rightarrow v \right\}.$$
Set also $\partial^Mf(\bar x):=\emptyset$ if $\bar x\not\in \dom f$. It follows from the definition that the following rubstness/closedness property of $\partial^Mf$ holds:
$$\left\{ v\in \Bbb R^n \colon \exists x^k\xrightarrow{f} \bar x,\ v^k\rightarrow v,\ v^k\in \partial^Mf(x^k)\right\}=\partial^Mf(\bar x).$$
Observe that from \cite[Theorem 8.6]{RWR}, one has $\partial^Ff(x)\subset\partial^Mf(x)$ for every $x\in \Bbb R^n$, where the first set is closed and convex while the second one is closed. If $f$ is differentiable at $\bar x$, then $\partial^Ff(\bar x)=\{\nabla f(\bar x)\}$, and if $f$ is continuously differentiable on a neighborhood of $\bar x$, then $\partial^Mf(\bar x)=\{\nabla f(\bar x)\}$. For convex function $f$, both the Fr\'{e}chet and limiting subdifferentials reduce to the classical subdifferential in the sense of convex analysis:
$$\partial f(\bar x)=\left\{v\in \Bbb R^n \colon \langle v, x-\bar x\rangle\leq f(x)-f(\bar x),\ \forall x\in \Bbb R^n\right\}.$$
It is necessary to mention another subdifferential named the {\it Clarke subdifferential} defined in \cite{CF1}, which was based on generalized directional derivatives, and it is also worth noting that Clarke subdifferential of a locally Lipschitz continuous function $f$ around $\bar x$ can be represented as the limiting subdifferential: $\partial^Cf(\bar x)=\co\partial^Mf(\bar x)$, where $\co\Omega$ denotes the convex hull of an arbitrary set $\Omega$.
\begin{proposition}
\cite[pp. 304]{RWR} Let $f=g+h,$ where $g$ is lower semicontinuous and $h$ is continuously differentiable on a neighborhood of $\bar x$. Then
$$\partial^Ff(\bar x)=\partial^Fg(\bar x)+\nabla h(\bar x)\ \hbox{ and } \ \partial^Mf(\bar x)=\partial^Mg(\bar x)+\nabla h(\bar x).$$
\end{proposition}

\begin{proposition}
\cite[pp. 422]{RWR} If a lower semicontinuous function $f:\Bbb R^n\rightarrow\bar{\Bbb R}$ has a local minimum at $\bar x\in dom f$, then $0\in\partial^Ff(\bar x)\subset\partial^Mf(\bar x)$. In the convex case, this condition is not only necessary for a local minimum but also sufficient for a global minimum.
\end{proposition}

Note that for a finite convex function $f$ on $\Bbb R^n$, if $y^k\in \partial f(x^k)$ for all $k$ and $\{ x^k\}$ is bounded, then the sequence $\{ y^k\}$ is also bounded, one can refer Definition 5.14, Proposition 5.15 and Theorem 9.13 in \cite{RWR}.

To establish our convergence results, we need the {\it Kurdyka--{\L}ojasiewicz property} (briefly say K{\L} property) defined as follows (see also \cite{PAB, HJPA, BAA}). Before that, let us recall the definitions of semi-algebraic set and function. A subset $\Omega$ of $\Bbb R^n$ is called {\it semi-algebraic} (see \cite{Ha2017}) if it can be represented as a finite union of sets of the form
$$\left\{x\in\Bbb R^n\colon p_i(x)=0,\ q_i(x)<0,\ \textrm{for all}\ i=1,\ldots,m \right\},$$
where $p_i$ and $q_i$ for $i=1,\ldots,m$ are polynomial functions. A function $f$ is said to be {\it semi-algebraic} if its graph is a semi-algebraic subset of $\Bbb R^{n+1}$.
\begin{definition}\label{pro3}{\rm
A lower semicontinuous function $f:\Bbb R^n\rightarrow\bar{\Bbb R}$ satisfies the {\it K{\L} property} at $x^*\in \dom\partial^Mf$ if there exist $\epsilon>0$, a neighborhood $U$ of $x^*$ and a continuous concave function $\theta:[0, \epsilon[\rightarrow[0, +\infty[$ with
\begin{itemize}
\item[{\bf (a)}] $\theta(0)=0,$ 
\item[{\bf (b)}] $\theta>0$ on $]0, \epsilon[,$ 
\item[{\bf (c)}] $\theta$ is of class $\mathcal{C}^1$ on $]0, \epsilon[,$
\item[{\bf (d)}] for every $x\in U$ and $f(x^*)<f(x)<f(x^*)+\epsilon,$ one has $$\theta'\left(f(x)-f(x^*)\right) d\left(0; \partial^Mf(x)\right)\geq 1.$$
\end{itemize}
}\end{definition}
It is known that a proper lower semicontinuous semi-algebraic function always satisfies the K{\L} property, one can see \cite{HJPA, BDLS,Ha2017}. Moreover, $f$ is said to satisfy the {\it strong Kurdyka--{\L}ojasiewicz property} at $x^*$ if {\bf (a)}$\sim${\bf (d)} hold for Clarke subdifferential $\partial^Cf(x)$. In fact, very recently, Bolte {\it et al.} \cite[Theorem 14]{BDLS} pointed out that the class of definable functions, which contains the class of semi-algebraic functions, satisfies the strong Kurdyka--{\L}ojasiewicz property at each point of $\dom\partial^Cf$.

In virtue of \cite[Lemma 2.1]{HJPA}, we know that a proper lower semicontinuous function $f:\Bbb R^n\rightarrow\bar{\Bbb R}$ has the K{\L} property at any point $\bar x\in \Bbb R^n$ such that $0\not\in\partial^Mf(\bar x)$. If the function $f:\Bbb R^n\rightarrow\Bbb R$ is differentiable and $\theta(t)=Mt^{\kappa}$, where $M>0$ and $\kappa\in[0, 1)$, then we have the following definition which is one special case of Definition \ref{pro3}.
\begin{definition}
{\rm
If $f:\Bbb R^n\rightarrow\Bbb R$ is a differentiable function, then $f$ is said to have the {\it {\L}ojasiewicz property} if for any critical point $\bar x$, there exist constants $M>0$, $\epsilon>0$ and $\kappa\in[0,1)$ such that
\begin{align*}
|f(x)-f(\bar x)|^{\kappa}\leq M\|\nabla f(x)\|,\;\; \hbox{for all }\; x\in\Bbb B(\bar x, \epsilon),  
\end{align*}
where we adopt the convention $0^0=1$ and the constant $\kappa$ is called {\it {\L}ojasiewicz exponent} of $f$ at $\bar x$.}
\end{definition}
The reader is referred to see \cite{LiPong2018} for some calculus of the {\L}ojasiewicz exponent for some classes of functions.
On the other hand, a differentiable function $f:\Bbb R^n\rightarrow\Bbb R$ is said to be {\it real analytic} if for every $x\in\Bbb R^n$, $f$ could be represented by a convergent power series in some neighbourhood of $x$. In addition, \cite{LS} showed that every real analytic function $f:\Bbb R^n\rightarrow\Bbb R$ satisfies the {\L}ojasiewicz property with exponent $\kappa\in[0,1)$.

We will also employ the following useful lemma to obtain bounds on the rate of convergence of the sequence generated by {\bf Algorithm \ref{al::main}} (see Section~\ref{Sec:3}). This lemma appears within the proof in \cite[Theorem 2]{AHBJ} for specific values of $\mu$ and $\nu$. For convenience, we give a brief proof, see also \cite[Lemma 3.1]{AFV} and \cite[Theorem 3.3]{ANT}.
\begin{lemma}\label{lem1}
Let $\{ t_k\}$ be a sequence in $\Bbb R_+$ and let $\mu$ and $\nu$ be some positive constants. Suppose that $t_k\rightarrow 0$ and the sequence satisfies
\begin{align}
t_k^{\mu}\leq \nu(t_k-t_{k+1}), \;\;\; \hbox{for all $k$ sufficiently large.} \label{ineq2}
\end{align}
Then 
\begin{itemize}
\item[{\bf (a)}]  If $\mu=0$, the sequence $\{t_k\}$ converges to $0$ in a finite number of steps.
\item[{\bf (b)}]  If $\mu\in (0, 1]$, the sequence $\{t_k\}$ converges linearly to $0$ with rate $1-\frac{1}{\nu}$.
\item[{\bf (c)}]  If $\mu>1$, there exists $\gamma>0$ such that for all $k$ sufficiently large $$t_k\leq \gamma k^{-\frac{1}{\mu-1}}.$$
\end{itemize}
\end{lemma}
\begin{proof}
{\bf (a)} If $\mu=0$, then from \eqref{ineq2}, one has $0\leq t_{k+1}\leq t_k-\frac{1}{\nu},$ which implies {\bf (a)}.

{\bf (b)} Now, we suppose that $\mu\in (0, 1]$. As $t_k\rightarrow 0$, we obtain that $t_k<1$ for all $k$ large enough, then from \eqref{ineq2},
$$t_k\leq t_k^{\mu}\leq\nu(t_k-t_{k+1}),$$
thus, $t_{k+1}\leq(1-\frac{1}{\nu})t_k$, which means that $\{t_k\}$ converges to $0$ with rate $1-\frac{1}{\nu}$ linearly.

{\bf (c)} Assume that $\mu>1$, if $t_k=0$ for some $k$, then from \eqref{ineq2}, we have $t_{k+1}=0$, which points out that the sequence $\{ t_k\}$ converges to zero in a finite number of steps and {\bf (c)} trivially holds. Therefore we can assume that $t_k>0, \forall k$.

Consider the decreasing function $u:(0, +\infty)\rightarrow\Bbb R$ defined by $u(t)=t^{-\mu}$. Assume that \eqref{ineq2} holds for all $k\geq N$, for some positive integer $N$, then for $k\geq N$, we get
$$\frac{1}{\nu}\leq(t_k-t_{k+1})u(t_k)\leq\int_{t_{k+1}}^{t_k}u(t)dt=\frac{t_k^{1-\mu}-t_{k+1}^{1-\mu}}{1-\mu}.$$
Since $\mu>1$, then
$$t_{k+1}^{1-\mu}-t_{k}^{1-\mu}\geq\frac{\mu-1}{\nu}, \;\;\hbox{for all } k\geq N.$$
Summing up $k$ from $N$ to $j-1\geq N$, one has
$$t_{j}^{1-\mu}-t_{N}^{1-\mu}\geq\frac{\mu-1}{\nu}(j-N),$$
which gives
$$t_j\leq\left(t_N^{1-\mu}+\frac{\mu-1}{\nu}(j-N)\right)^{\frac{1}{1-\mu}}, \;\;\hbox{ for all } j\geq N+1.$$
As a conclusion, there exists $\gamma>0$ such that $t_k\leq \gamma k^{-\frac{1}{\mu-1}}$ for all $k$ sufficiently large.
\end{proof}

We close this section by recalling the well known descent lemma for a smooth function with Lipschitz continuous gradient, which says that it can be upper bounded by a certain quadratic function; see, \cite[Lemma 5.7]{Beck2017} for the  proof (see also \cite{N,OR}).
\begin{lemma}\label{lem2}
Suppose that $g: \Bbb R^n \rightarrow \Bbb R$ is a differentiable function and $\Omega$ is a nonempty convex subset of $\Bbb R^n$ such that $\nabla g$ satisfies the Lipschitz condition on $\Omega$ with constant $L$. Then for all $x, y \in \Omega$ we have:
\begin{align*}
g(y)\leq g(x)+\left\langle\nabla g(x), y-x \right\rangle+\frac{L}{2}\|y-x\|^2. 
\end{align*}
\end{lemma}

\section{A proximal-type algorithm with linesearch}\label{Sec:3}

In this section, we introduce a proximal point algorithm with linesearch to solve the problem \eqref{problem}, and establish its convergence and convergent rate.
%
%
%

\begin{framed}
\begin{algorithm}[Solving $P(\varphi, g, h)$]\label{al::main} 
$\quad$\\
{\sf
{\bf Initialization}. Pick $x^0\in\Bbb R^n$, choose parameters $\eta\in (0,1)$, $\alpha>0$, $\{ \lambda_k\}\subset (0, +\infty)$.\\
{\bf Iteration $k$} ($k = 0, 1, 2, \ldots$).  Having $x^k$ do the following steps:

{\sc Step 1}. Solve the following strongly convex program
$$\min\limits_{x\in{\Bbb R^n}}\left\{g(x)-\langle\nabla h(x^k)-\nabla\varphi(x^k), x-x^k\rangle+\frac{\lambda_k}{2}\|x-x^k\|^2 \right\},$$ to get the unique solution $y^k$.

Set $d^k=y^k-x^k$. If $d^k=0$, then stop. Otherwise, go to {\sc Step 2}.

{\sc Step 2}. (Armijo  linesearch rule)
\begin{align}
&\hbox{Find $m_k$ as the smallest positive integer number $m$ such that \qquad\qquad\quad}\nonumber\\
& \qquad\qquad \qquad f(y^k+\eta^m  d^k) \leq f(y^k)-\alpha\eta^m\|d^k\|^2. \nonumber
\end{align}
Set $\eta_k=\eta^{m_k}$ and $x^{k+1}=y^k+\eta_kd^k$ and go to {\bf Iteration $k$} with $k$ replaced by $k+1$.

}
\end{algorithm}
\end{framed}

For the above proposed Algorithm \ref{al::main}, we have the following results.
\begin{theorem}\label{theo1}
Suppose that $\Omega$ is a convex set containing $\{x^k\}, \{y^k\}$ such that $\nabla\varphi $ is Lipschitz continuous with constant $L_1$ on $\Omega$ and $ L_1+2\hat{\lambda}\leq\lambda_k\leq\bar{\lambda}$ where $\hat{\lambda}>0,$ then
\begin{itemize}
\item[{\bf (a)}]   $f(y^k)\leq f(x^k)-\frac{\lambda_k-L_1}{2}\|y^k-x^k\|^2.$
  \item[{\bf (b)}]  The linesearch is well defined.
  \item[{\bf (c)}]  $f(x^{k+1}) \leq  f(x^k) - \left(\frac{\lambda_k - L_1}{2} + \alpha\eta_k \right)\|d^k\|^2,$ the sequence $\{f(x^k)\}$ is strictly decreasing and convergent.
  \item[{\bf (d)}]  Any accumulation point of $\{x^k\}$ is a stationary point of $f$.
  \item[{\bf (e)}]  $\sum\limits_{k=0}^{\infty}\| d^k\|^2<\infty$ and $\sum\limits_{k=0}^{\infty}\| x^{k+1}-x^k\|^2<\infty$.
\end{itemize}
\end{theorem}
\begin{proof}
{\bf (a)} From Algorithm \ref{al::main}, we have
$$ g(x^k)\geq g(y^k)- \left\langle\nabla h(x^k), y^k-x^k \right\rangle + \left\langle\nabla\varphi(x^k), y^k-x^k \right\rangle+\frac{\lambda_k}{2}\| y^k-x^k\|^2.$$
In addition, by convexity of $h$,
$$h(y^k)\geq h(x^k)+\left\langle\nabla h(x^k), y^k-x^k \right\rangle,$$
and $\nabla\varphi(x)$ is Lipschitz with constant $L_1$, by Lemma \ref{lem2},
$$\varphi (y^k)\leq \varphi(x^k)+ \left\langle\nabla\varphi(x^k), y^k-x^k \right\rangle+\frac{L_1}{2}\|y^k-x^k\|^2.$$
Combining the above inequalities, one has
$$f(y^k)\leq f(x^k)-\frac{\lambda_k-L_1}{2}\|y^k-x^k\|^2.$$

{\bf (b)} We prove by contradiction. If for all $m\geq 1$, we have
\begin{align}
&f(y^k+\eta^md^k)>f(y^k), \;\; \forall m \nonumber\\
\Leftrightarrow\quad& f(y^k+\eta^md^k)-f(y^k)>0, \;\; \forall m \nonumber\\
\Rightarrow\quad& \left\langle\nabla f(y^k),\  \eta^md^k \right\rangle+o(\eta^m)>0, \;\; \forall m \nonumber\\
\Leftrightarrow\quad& \left\langle\nabla f(y^k),\  d^k \right\rangle + \frac{o(\eta^m)}{\eta^m}>0 \nonumber\\
\Rightarrow\quad& \left\langle\nabla f(y^k),\ d^k \right\rangle>0. \label{jia}
\end{align}
Beside that, from {\sc Step 1} in Algorithm \ref{al::main}, we have
$$\nabla g(y^k)-\nabla h(x^k)+\nabla\varphi(x^k)+\lambda_k(y^k-x^k)=0,$$
and
$$\nabla f(y^k)=\nabla\varphi(y^k)+\nabla g(y^k)-\nabla h(y^k),$$
then
$$\nabla f(y^k)=-\nabla h(y^k)+\nabla h(x^k)+\nabla\varphi (y^k)-\nabla\varphi(x^k)-\lambda_k(y^k-x^k).$$
Hence
\begin{align*}
\langle\nabla f(y^k), d^k\rangle &=-\langle\nabla h(y^k)-\nabla h(x^k), d^k\rangle+\langle\nabla\varphi(y^k)-\nabla\varphi(x^k), d^k\rangle-\lambda_k\|y^k-x^k\|^2\\
& \leq -(\lambda_k-L_1)\|y^k-x^k\|^2<0.
\end{align*}
This contradicts to \eqref{jia}. So the linesearch is well defined.

{\bf (c)} From {\sc Step 2} in Algorithm \ref{al::main}, one has,
\begin{align}
f(x^{k+1})&\leq f(y^k)-\alpha\eta_k\|d^k\|^2  \nonumber\\
&\leq f(x^k)-\frac{\lambda_k-L_1}{2}\|y^k-x^k\|^2-\alpha\eta_k\|d^k\|^2 \nonumber\\
&= f(x^k)-(\frac{\lambda_k-L_1}{2}+\alpha\eta_k)\|d^k\|^2. \label{ineq3}
\end{align}
Hence $\{f(x^k)\}$ is a strictly decreasing sequence, and combining with $\inf\limits_{x\in\Bbb R^n}f(x)>-\infty$, we get $\lim\limits_{k\rightarrow\infty}f(x^k)$ does exist.

{\bf (d)} Since $\{f(x^k)\}$ is convergent, then $$f(x^{k+1})-f(x^k)\rightarrow 0,$$
from \eqref{ineq3}, one has $$\|d^k\|^2=\|y^k-x^k\|^2\rightarrow 0.$$
Let $x^*$ be any accumulation point of $\{ x^k\}$ and let $x^{k_i}$ be a subsequence of $\{x^k\}$ converging to $x^*$. Since $\| y^{k_i}-x^{k_i}\|\rightarrow 0$, one has
$$\lim\limits_{i\rightarrow\infty}y^{k_i}= x^*.$$
The {\sc Step 1} in Algorithm \ref{al::main} yields
$$\nabla g(y^{k_i})-\left(\nabla h(x^{k_i})-\nabla\varphi (x^{k_i})\right)+\lambda_k\left( y^{k_i}-x^{k_i} \right)=0,$$
letting $i\rightarrow\infty$, we get from the above inequality that
$$\nabla g(x^*)-\nabla h(x^*) +\nabla\varphi(x^*)=0,$$
which means that $x^*$ is a stationary point of $f$.

{\bf (e)} Observe, from \eqref{ineq3}, one has
\begin{align}
\left(\frac{\lambda_k-L_1}{2}+\alpha\eta_k \right)\|d^k\|^2\leq f(x^k)-f(x^{k+1}).\label{ineq4}
\end{align}
Summing up \eqref{ineq4} from $0$ to $N$, we get
$$\sum\limits_{k=0}^{N}\left(\frac{\lambda_k-L_1}{2}+\alpha\eta_k \right)\|d^k\|^2\leq f(x^0)-f(x^{N+1})\leq f(x^0)-\inf\limits_{x\in\Bbb R^n}f(x),$$
since $\lambda_k\geq L_1+2\widehat{\lambda}$, then taking the limit as $N\rightarrow\infty$, the above inequality becomes
$$\sum\limits_{k=0}^{\infty}\widehat{\lambda}\|d^k\|^2\leq\sum\limits_{k=0}^{\infty}\left(\frac{\lambda_k-L_1}{2}+\alpha\eta_k \right)\|d^k\|^2\leq f(x^0)-\inf\limits_{x\in\Bbb R^n}f(x)<\infty.$$
Thus
$$\sum\limits_{k=0}^{\infty}\|d^k\|^2<\infty.$$
In addition, $$x^{k+1}-x^k=y^k+\eta_kd^k-x^k=(1+\eta_k)d^k,$$
so
$$\sum\limits_{k=0}^{\infty}\|x^{k+1}-x^k\|^2=\sum\limits_{k=0}^{\infty}(1+\eta_k)^2\|d^k\|^2\leq\sum\limits_{k=0}^{\infty}(1+\eta)^2\|d^k\|^2<\infty.$$
Therefore, the proof is complete.
\end{proof}

\begin{remark}{\rm
In a recent paper  \cite{AN}, An {\it et al.} proposed the following proximal point algorithm:
%
%
\begin{framed}
\begin{algorithm}\label{al::2} 
$\quad$\\
{\sf
{\bf Initialization}. Pick $x^0 \in \Bbb R^n$ and a tolerance $\epsilon > 0.$ Fixed  $ \lambda > L_1$.\\
{\bf Iteration $k$} ($k = 0, 1, 2, \ldots$).  Having $x^k$ do the following steps:

{\sc Step 1}. Solve the following strongly convex program
$$\min\limits_{x\in{\Bbb R^n}}\left\{g(x)- \left\langle\nabla h(x^k)-\nabla\varphi(x^k), x-x^k \right\rangle+\frac{\lambda}{2}\|x-x^k\|^2 \right\},$$ 
 to get the unique solution $x^{k+1}$.

{\sc Step 2}. If $\|x^{k+1} - x^k\| \leq \epsilon$, then stop. Otherwise, increase $k$ by $1$ and go to {\sc Step 1}.
}
\end{algorithm}
\end{framed}
The main difference of  Algorithm~\ref{al::main} and Algorithm~\ref{al::2} is at {\sc Step 2}. In Algorithm~\ref{al::main}, we use $d^k$ as the descent direction at $y^k$. In addition, $x^{k+1} = y^k + \eta_kd^k = x^k+ (1+\eta_k)d^k$ and
$$f(y^k)\leq f(x^k)-\frac{\lambda_k-L_1}{2}\| d^k\|^2,$$
and
$$f(x^{k+1}) \leq f(y^k)-\alpha\eta_k\| d^k\|^2 \leq f(x^k)-\left(\frac{\lambda_k-L_1}{2}+\alpha\eta_k \right)\| d^k\|^2.$$
Hence, we go a longer step at $x^k$ and make the value of function $f$ decrease much more than that in the proximal point algorithm. \qed
}\end{remark}

The following theorem establishes the convergence and the convergent rate of Algorithm~\ref{al::main}.
\begin{theorem}
Under assumptions of {\rm Theorem~\ref{theo1}} and we further assume that $\nabla g$ is locally Lipschitz continuous with constant $L_2$ and $f$ satisfies the {\L}ojasiewicz inequality with exponent $\kappa\in [0, 1)$. If the sequence $\{x^k\}$ has a limit point $x^*,$ then the whole sequence $\{ x^k\}$ converges to $x^*,$ which is a stationary point of $f$. Moreover denoting $f^*:=f(x^*),$ the following estimations hold$:$
\begin{itemize}
  \item[{\bf (a)}]  If $\kappa=0,$ then the sequences $\{x^k\}$ and $\{f(x^k)\}$ converge in a finite number of steps to $x^*$ and $f^*,$ respectively.
  \item[{\bf (b)}]  If $\kappa\in(0, \frac{1}{2}],$ then the sequences $\{x^k\}$ and $\{f(x^k)\}$ converge linearly to $x^*$ and $f^*,$ respectively.
  \item[{\bf (c)}]  If $\kappa\in(\frac{1}{2}, 1),$ then there exist some positive constants $A_1$ and $A_2$ such that for $k$ large enough$,$
$$\|x^k-x^*\|\leq A_1k^{-\frac{1}{2\kappa-1}},\;\hbox{ and }\; f(x^k)-f^*\leq A_2{k^{-\frac{1}{2\kappa-1}}}.$$
\end{itemize}
\end{theorem}
\begin{proof}
By Theorem \ref{theo1} {\bf (c)}, we have $\lim\limits_{k\rightarrow\infty}f(x^k)=f^*$. If $x^*$ is a limit point of $\{x^k\}$, then there exists a subsequence $\{x^{k_i}\}$ of $\{ x^k\}$ which converges to $x^*$. By continuity of $f$, we have that
$$f(x^*)=\lim\limits_{i\rightarrow\infty}f(x^{k_i})=\lim\limits_{k\rightarrow\infty}f(x^k)=f^*.$$
Hence $f$ is finite and has the same value $f^*$ at every limit point of $\{x^k\}$. If $f(x^k)=f^*$ for some $k>1$, then $f(x^k)=f(x^{k+r}),\;\forall  r\geq 0$, since the sequence $\{f(x^k)\}$ is decreasing. Therefore, $x^k=x^{k+r}$ for all $r\geq 0$ and Algorithm \ref{al::main} terminates after a finite number of steps.

Now we assume that $f(x^k)>f^*$, $\forall k$. Since $f$ satisfies the {\L}ojasiewicz property, there exist $M>0$, $\epsilon_1>0$ and $\kappa\in[0, 1)$ such that
\begin{align}
|f(x)-f(x^*)|^{\kappa}\leq M\|\nabla f(x)\|,\;\; \forall \ x\in\Bbb B(x^*, \epsilon_1). \label{99}
\end{align}
Further, as $\nabla g$ is Locally Lipschitz around $x^*$, there exist some constants $L_2> 0$ and $\epsilon_2>0$ such that
\begin{align}
\|\nabla g(x)-\nabla g(y)\| \leq L_2\|x-y\|,\;\; \forall \ x, y \in\Bbb B(x^*, \epsilon_2). \label{98}
\end{align}
Let $\epsilon:=\frac{1}{2}\min\{\epsilon_1, \epsilon_2\},$ clearly $\epsilon > 0.$

Since $\lim\limits_{i\rightarrow\infty}x^{k_i}=x^*$ and $\lim\limits_{i\rightarrow\infty}f(x^{k_i})=f^*$, we can find an index $k_{\epsilon}$ large enough such that
\begin{align}
\|x^{k_{\epsilon}}-x^*\|+\frac{2M(L_2+\lambda_k)}{(1-\kappa)(\lambda_k-L_1)}(f(x^{k_{\epsilon}})-f^*)^{1-\kappa}<\epsilon. \label{97}
\end{align}
By Theorem \ref{theo1} {\bf (c)}, we know that $d^k=y^k-x^k\rightarrow 0$ as $k\rightarrow\infty$. Then without loss of generality, we can assume that
$$\|y^k-x^k\|\leq\epsilon,\;\;\forall k\geq {k_{\epsilon}}.$$

We now claim that, $\forall k\geq {k_{\epsilon}}$, whenever $x^k\in\Bbb B(x^*, \epsilon)$ the following holds
\begin{align}
\|x^{k+1}-x^k\|\leq\frac{M(L_2+\lambda_k)(1+\eta_k)}{(1-\kappa)\left(\frac{\lambda_k-L_1}{2}+\alpha\eta_k \right)}\left[(f(x^k)-f^*)^{1-\kappa}-\left(f(x^{k+1})-f^*\right)^{1-\kappa}\right]. \label{96}
\end{align}
Indeed, consider the concave function $\theta:(0, +\infty)\rightarrow(0,+\infty)$ defined by $\theta(t)=t^{1-\kappa}$. Then we have
$$\theta(t_1)-\theta(t_2)\geq \nabla \theta(t_1)^T(t_1-t_2),\;\;\forall t_1, t_2>0.$$
Substituting in this inequality $t_1$ by $\left(f(x^k)-f^*\right)$ and $t_2$ by $\left(f(x^{k+1})-f^*\right)$ and using \eqref{99} and then \eqref{ineq4}, we have
\begin{align}
&\ \left(f(x^k)-f^*\right)^{1-\kappa}- \left(f(x^{k+1})-f^* \right)^{1-\kappa}\nonumber\\ 
\geq&\ \frac{1-\kappa}{\left(f(x^k)-f^*\right)^{\kappa}} \left(f(x^k)-f(x^{k+1})\right)\nonumber\\
\geq&\ \frac{1-\kappa}{M\|\nabla f(x^k)\|} \left(\frac{\lambda_k-L_1}{2}+\alpha\eta_k \right) \|d^k\|^2\nonumber\\
=&\ \frac{1-\kappa}{M\|\nabla f(x^k)\|}   \frac{\left(\frac{\lambda_k-L_1}{2}+\alpha\eta_k \right)}{(1+\eta_k)^2} \|x^{k+1}-x^k\|^2. \label{95}
\end{align}
On the other hand, by Algorithm \ref{al::main}, we get
$$\nabla g(y^k)-\nabla h(x^k)+\nabla\varphi(x^k)+\lambda_k(y^k-x^k)=0$$
and
$$\|y^k-x^*\|\leq\|y^k-x^k\|+\|x^k-x^*\|\leq 2\epsilon\leq\epsilon_2.$$
Using \eqref{98}, we obtain
\begin{align}
\|\nabla f(x^k)\|\ =&\ \|\nabla\varphi(x^k)+\nabla g(x^k)-\nabla h(x^k)\|\nonumber\\
=&\ \|\nabla g(x^k)-\nabla g(y^k)-\lambda_k(y^k-x^k) \|\nonumber\\
\leq&\ \|\nabla g(x^k)-\nabla g(y^k)\|+\lambda_k\|y^k-x^k\|\nonumber\\
\leq&\ (L_2+\lambda_k) \|x^k-y^k\|\nonumber\\
=&\ \frac{L_2+\lambda_k}{1+\eta_k} \|x^{k+1}-x^k\|. \label{94}
\end{align}
From \eqref{95} and \eqref{94}, we get
$$\left(f(x^k)-f^*\right)^{1-\kappa}-\left(f(x^{k+1})-f^* \right)^{1-\kappa} \geq \frac{1-\kappa}{M \left(\frac{L_2+\lambda_k}{1+\eta_k}\right)\|x^{k+1}-x^k\|} \frac{\frac{\lambda_k-L_1}{2}+\alpha\eta_k}{(1+\eta_k)^2} \|x^{k+1}-x^k\|^2,$$
which yields \eqref{96}.

It then follows from \eqref{96} that
\begin{align}
\|x^{k+1}-x^k\|\leq&\ \frac{M(L_2+\lambda_k)(1+\eta)}{(1-\kappa)\left(\frac{\lambda_k-L_1}{2}\right)}\left[\left(f(x^k)-f^*\right)^{1-\kappa}-\left(f(x^{k+1})-f^*\right)^{1-\kappa}\right]\nonumber\\
=&\ \frac{2M(L_2+\lambda_k)(1+\eta)}{(1-\kappa)(\lambda_k-L_1)}\left[\left(f(x^k)-f^*\right)^{1-\kappa}-\left(f(x^{k+1})-f^*\right)^{1-\kappa}\right] \label{93}
\end{align}
for all $k\geq {k_{\epsilon}}$ such that $x^k\in\Bbb B(x^*, \epsilon)$.  

We next prove that $x^k\in \Bbb B(x^*, \epsilon)$, $\forall k\geq {k_{\epsilon}}$ by induction. Indeed, from \eqref{97} the claim holds for $k={k_{\epsilon}}$.  
We suppose that it also holds for $k={k_{\epsilon}}, {k_{\epsilon}}+1, \ldots, {k_{\epsilon}}+r-1$, with $r\geq 1$. Then \eqref{93} is valid for $k={k_{\epsilon}}, {k_{\epsilon}}+1, \ldots, {k_{\epsilon}}+r-1$. 
Therefore
\begin{align*}
 & \ \|x^{{k_{\epsilon}}+r}-x^*\| \\
=&\ \|x^{{k_{\epsilon}}+1}-x^{k_{\epsilon}}+x^{{k_{\epsilon}}+2}-x^{{k_{\epsilon}}+1}+\cdots+x^{{k_{\epsilon}}+r}-x^{{k_{\epsilon}}+r-1}+x^{k_{\epsilon}}-x^*\|\\
\leq&\ \|x^{k_{\epsilon}}-x^*\|+\sum\limits_{i=1}^{r}\|x^{{k_{\epsilon}}+i}-x^{{k_{\epsilon}}+i-1}\|\\
\leq&\ \|x^{k_{\epsilon}}-x^*\|+\frac{2M(L_2+\lambda_k)(1+\eta)}{(1-\kappa)(\lambda_k-L_1)} \sum\limits_{i=1}^{r}\left[\left(f(x^{{k_{\epsilon}}+i-1})-f^*\right)^{1-\kappa}- \left(f(x^{{k_{\epsilon}}+i})-f^*\right)^{1-\kappa}\right]\\
\leq&\ \|x^{k_{\epsilon}}-x^*\|+\frac{2M(L_2+\lambda_k)}{(1-\kappa)(\lambda_k-L_1)}\left(f(x^{k_{\epsilon}})-f^*\right)^{1-\kappa} <\epsilon.
\end{align*}
Now adding \eqref{93} from $k={k_{\epsilon}}$ to ${k_{\epsilon}}+r-1$, one has
\begin{align}
\sum\limits_{k={k_{\epsilon}}}^{{k_{\epsilon}}+r-1}\|x^{k+1}-x^k\|\leq\frac{2M(L_2+\lambda_k)(1+\eta)}{(1-\kappa)(\lambda_k-L_1)}\left(f(x^{k_{\epsilon}})-f^*\right)^{1-\kappa}. \label{92}
\end{align}
Taking the limit as $r\rightarrow\infty$, we can conclude from \eqref{92} that
\begin{align}
\sum\limits_{k=1}^{\infty}\|x^{k+1}-x^k\|<\infty, \label{91}
\end{align}
which means that $\{x^k\}$ is a Cauchy sequence. Therefore, combining with $x^*$ is a limit point of $\{x^k\}$, we conclude that the whole sequence $\{x^k\}$ converges to $x^*$. By Theorem \ref{theo1} {\bf (d)}, $x^*$ must be a stationary point of $f$.

For $k\geq N$, it follows from \eqref{99}, \eqref{98} and then \eqref{ineq4}, that
\begin{align}
\left(f(x^k)-f^*\right)^{2\kappa}\ \leq&\ M^2\|\nabla f(x^k)\|^2\nonumber\\
\leq&\ M^2\|\nabla\varphi(x^k)+\nabla g(x^k)-\nabla h(x^k)\|^2\nonumber\\
=&\ M^2\|\nabla g(x^k)-\nabla g(y^k)-\lambda_k( y^k-x^k) \|^2\nonumber\\
\leq&\ M^2(L_2+\lambda_k)^2\|x^k-y^k\|^2\nonumber\\
\leq&\ \frac{M^2(L_2+\lambda_k)^2}{\frac{\lambda_k-L_1}{2}+\alpha\eta_k}  \left(f(x^k)-f(x^{k+1})\right)\nonumber\\
\leq&\ C \left[\left(f(x^k)-f^*\right)-\left(f(x^{k+1})-f^*\right)\right], \label{90}
\end{align}
where $C=\frac{M^2(L_2+\bar{\lambda})^2}{\hat{\lambda}+\alpha\eta}$.
By applying Lemma \ref{lem1} with $t_k=f(x^k)-f^*$, $\mu=2\kappa$ and $\nu=C$, {\bf (a)}$\sim${\bf (c)} regarding the sequence $\{f(x^k)\}$ follow from \eqref{90}.

By \eqref{91}, we know that $\mathcal{R}_i=\sum\limits_{k=i}^{\infty}\|x^{k+1}-x^k\|$ is finite. Note that $\|x^i-x^*\|\leq \mathcal{R}_i$ by the triangle inequality. Therefore, the rate of convergence of $x^i$ to $x^*$ can be deduced from the convergence rate of $\mathcal{R}_i$ to $0$. Adding \eqref{93} from $i$ to $r$ with ${k_{\epsilon}}\leq i\leq r$, we have
\begin{align*}
\mathcal{R}_i=&\ \lim\limits_{r\rightarrow\infty}\sum\limits_{k=i}^r\|x^{k+1}-x^k\| \\
\leq&\ \frac{2M(L_2+\lambda_i)(1+\eta)}{(1-\kappa)\hat{\lambda}} \left(f(x^i)-f^*\right)^{1-\kappa}\\
\leq& T_1\left(f(x^i)-f^*\right)^{1-\kappa},
\end{align*}
where $T_1:=\frac{2M(L_2+\bar\lambda)(1+\eta)}{(1-\kappa)\hat{\lambda}}>0$. Together with \eqref{99} and \eqref{94}, we have
\begin{align*}
\mathcal{R}_i^{\frac{\kappa}{1-\kappa}}\ \leq&\ T_1^{\frac{\kappa}{1-\kappa}}|f(x^i)-f^*|^{\kappa}\\
\leq&\ MT_1^{\frac{\kappa}{1-\kappa}}\|\nabla f(x^i)\|\\
\leq&\ MT_1^{\frac{\kappa}{1-\kappa}}\left(\frac{L_2+\lambda_i}{1+\eta_i}\right)\|x^{i+1}-x^i\|\\
\leq&\ MT_1^{\frac{\kappa}{1-\kappa}}(L_2+\lambda_i)\|x^{i+1}-x^i\|\\
\leq&\ M\left(L_2+\bar{\lambda}\right)T_1^{\frac{\kappa}{1-\kappa}}(\mathcal{R}_i-\mathcal{R}_{i+1})
\end{align*}
Hence, by setting $T_2:= M(L_2+\bar{\lambda})T_1^{\frac{\kappa}{1-\kappa}}$ that is clearly positive, the above inequality becomes
$$\mathcal{R}_i^{\frac{\kappa}{1-\kappa}}\leq T_2(\mathcal{R}_i-\mathcal{R}_{i+1}).$$
Now let $\mu =\frac{\kappa}{1-\kappa}$, $\nu= T_2$, and applying Lemma \ref{lem1}, we conclude that the statements in {\bf (a)}$\sim${\bf (c)} regarding the sequence $\{x^k\}$ hold.
\end{proof}

\section{Convergence of the Inertial Proximal Algorithm}\label{Sec:4}

In this section, we consider the problem \eqref{problem}, where $g$ and $h$ are not necessary differentiable. In \cite{MM}, P.E. Maing\'{e} and A. Moudafi introduced the following inertial proximal algorithm for solving \eqref{problem}.
{
%
%
%

\begin{framed}
\begin{algorithm}\label{al::3} 
$\quad$\\
{\sf
{\bf Initialization}. $x^0, y^0\in\Bbb R^n$; $\lambda, \mu>0$; $\alpha+\beta>0$; $\beta>0$; $\gamma>0$ and $\tau>-\frac{2+\alpha}{2\beta}$.\\
{\bf Iteration $k$} ($k = 0, 1, 2, \ldots$).  Having $x^k$, $y^k$ do the following steps:

{\sc Step 1}. Compute $q^k\in\partial h(x^k)$ and $\nabla\varphi (x^k).$

{\sc Step 2}. $x^{k+1}=(I+\lambda\partial g)^{-1}\left(x^k-\lambda \left(\nabla\varphi(x^k)-q^k \right)-\mu\left(\alpha x^k+\beta y^k\right)\right).$

{\sc Step 3}. Compute $y^{k+1}=y^k-\frac{1}{\rho}\left[\alpha x^k+\beta y^k+\gamma\alpha(x^{k+1}-x^k)\right]$, where $\rho=1+\tau\beta+\frac{\alpha+\beta}{2}$, and go to {\bf Iteration $k$} with  $k$ being replaced by $k+1$.
}
\end{algorithm}
\end{framed}

Set
$a=\frac{2\alpha}{2+\alpha+\beta},$  $b=\frac{2\beta}{2+\alpha+\beta}$ and for $\delta>0$,
consider the discrete energy $E_{k}(\delta)$ of Algorithm~\ref{al::2} defined by
$$E_{k}(\delta)=\delta f(x^k)+\frac{1}{2}\|ax^k+by^k\|^2.$$

The following theorem was proved in \cite{MM}.
\begin{theorem}\label{theo3}
\cite[Theorem 3.2]{MM} Assume that $\beta>0,$ $\alpha+\beta>0;$ $\tau>-\frac{2+\alpha}{2\beta};$
$\gamma\geq\frac{1}{2}$ and $g,$ $h$ are proper lower semicontinuous convex functions on $\Bbb R^n,$ $\varphi$ is a differentiable function on $\Bbb R^n$ with $L_1$-Lipschitz continuous gradient$,$ for some $L_1\in (0, +\infty)$ and $\lambda, \mu$ verify
$$\lambda L_1+\mu(\gamma\alpha+\rho)\leq 1.$$
Then $\{ x^k\}$ and $\{ y^k\}$ generated by Algorithm \ref{al::3} satisfy the following properties:
\begin{itemize}
 \item[{\bf (a)}]  For $\delta\in [\frac{\lambda}{\rho\mu}(\sqrt{b_2}-\sqrt{a_2+b_2})^2, \frac{\lambda}{\rho\mu}(\sqrt{b_2}+\sqrt{a_2+b_2})^2],$ where $a_2=a\left[1+2b(\gamma-\frac{1}{2})\right],$ $b_2=b\left[1+b(\tau-\frac{1}{2})\right],$ the energy $\{E_k(\delta)\}$ is a decreasing and converging sequence.
  \item[{\bf (b)}]  $\lim\limits_{k\rightarrow\infty}f(x^k)$ exists.
  \item[{\bf (c)}]  $\lim\limits_{k\rightarrow\infty}\|x^{k+1}-x^k\|=\lim\limits_{k\rightarrow\infty}\|y^{k+1}-y^k\|=0.$
  \item[{\bf (d)}]  $\lim\limits_{k\rightarrow\infty}\|\alpha x^{k}+\beta y^k\|=0.$
  \item[{\bf (e)}]  If $\{ x^k\}$ and $\{ q^k\}$ are bounded$,$ then every cluster point $x^*$ of the sequence $\{ x^k\}$ is a critical point of the function $f$.
\end{itemize}
\end{theorem}

When $h$ is differentiable, Algorithm~\ref{al::3} becomes the following algorithm.
%
%
%

\begin{framed}
\begin{algorithm}\label{al::4} 
$\quad$\\
{\sf
{\bf Initialization}. $x^0, y^0\in\Bbb R^n$; $\lambda, \mu>0$; $\alpha+\beta>0$; $\beta>0$; $\gamma>0$ and $\tau>-\frac{2+\alpha}{2\beta}$.\\
{\bf Iteration $k$} ($k = 0, 1, 2, \ldots$).  Having $x^k$, $y^k$ do the following steps:

{\sc Step 1}. Compute $\nabla\varphi (x^k)$ and $\nabla h(x^k)$.

{\sc Step 2}. Solve the following strongly convex program
\begin{align*}
\min\limits_{x\in\Bbb R^n}\left\{g(x)-\left\langle\nabla h(x^k)-\nabla\varphi(x^k), x-x^k \right\rangle+\frac{\mu}{\lambda}\left\langle\alpha x^k+\beta y^k, x-x^k\right\rangle+\frac{1}{\lambda}\|x-x^k\|^2 \right\}
\end{align*}
to get the unique solution $x^{k+1}$.

{\sc Step 3}. $y^{k+1}=y^k-\frac{1}{\rho}\left[\alpha x^k+\beta y^k+\gamma\alpha(x^{k+1}-x^k)\right]$, where $\rho=1+\tau\beta+\frac{\alpha+\beta}{2},$ and go to {\bf Iteration $k$} with  $k$ being replaced by $k+1$.
}
\end{algorithm}
\end{framed}

Set
$\delta_1=(a_2+b_2)\beta$ and $z=(x,y)\in\Bbb R^n\times\Bbb R^n$, define the function $\phi:\Bbb R^n\times\Bbb R^n\rightarrow\Bbb R$ as follows
$$\phi(z)=\phi(x, y)=\delta_1 f(x)+\frac{1}{2}\|ax+by\|^2.$$

The following theorem establishes the convergence of $\{ x^k\}$.
\begin{theorem}
Under assumptions of {\rm Theorem~\ref{theo3},} we further assume that $\nabla h$ is Lipschitz continuous with constant $L_3$ and $\phi(z)$ is a K{\L} function$,$ then
\begin{itemize}
  \item[{\bf (a)}]  $\displaystyle{\phi(z^{k+1})\leq\phi(z^k)-\frac{1}{2}(2a_2+b_2)\|x^{k+1}-x^k\|^2-\frac{b_2}{2}\|y^{k+1}-y^k\|^2}$
$$-\frac{b_2}{2}\|(x^{k+1}-x^k)-(y^{k+1}-y^k)\|^2,$$
where $z^k=(x^k, y^k),$ $\forall k$.
  \item[{\bf (b)}]  If $\{ x^k\}$ has a limit point $x^*,$ then the whole sequence $\{ x^k\}$ converges to $x^*$.
  \item[{\bf (c)}] If in addition the function $\theta$ in the K{\L} inequality has the form $\theta (t) = M t^{1-\kappa},$ then we have the following
      \begin{itemize}
        \item[(i)] if $\kappa  = 0,$ then the algorithm terminate in a finite steps$;$
        \item[(ii)] if $0 < \kappa \leq \frac{1}{2},$ then there exist $A_1 > 0$ and $\zeta \in (0, 1)$ such that
                                   $$\| x^k - x^* \| \leq A_1\zeta^k;$$
        \item[(iii)] if $\frac{1}{2} < \kappa < 1,$ then there exists $A_2 > 0$ such that
                                     $$\| x^k - x^* \| \leq A_2 k^{\frac{1-\kappa}{1-2\kappa}}.$$
      \end{itemize}
\end{itemize}
\end{theorem}

\begin{proof}
{\bf (a)} From the inequality (3.14) in \cite{MM}, we have
\begin{align}
E_{k+1}(\delta)-E_k(\delta)+\delta \frac{\rho\mu}{\lambda}\| x^{k+1}-x^k\|^2+b_2\| y^{k+1}-y^k\|^2& \nonumber\\
+\left(a_2-\delta \frac{\rho\mu}{\lambda}\right)\left\langle x^{k+1}-x^k, y^{k+1}-y^k \right\rangle&\leq 0. \label{ineq41}
\end{align}
Replacing 
$$\left\langle x^{k+1}-x^k, y^{k+1}-y^k \right\rangle = -\frac{1}{2}\|(x^{k+1}-x^k)-(y^{k+1}-y^k)\|^2+\frac{1}{2}\|x^{k+1}-x^k\|^2+\frac{1}{2}\|y^{k+1}-y^k\|^2$$ into \eqref{ineq41}, 
we get
\begin{align}
E_{k+1}(\delta)-E_k(\delta)+\frac{1}{2}\left(a_2+\delta \frac{\rho\mu}{\lambda}\right)\|x^{k+1}-x^k\|^2 + \left(b_2+\frac{a_2}{2}-\delta \frac{\rho\mu}{2\lambda}\right)\|y^{k+1}-y^k\|^2&\nonumber\\
+\frac{1}{2}\left(\delta \frac{\rho\mu}{\lambda}-a_2 \right)\|(x^{k+1}-x^k)-(y^{k+1}-y^k)\|^2&\leq 0. \label{42}
\end{align}
In particular, when $\delta=\delta_1=(a_2+b_2) \frac{\lambda}{\rho\mu}$, \eqref{42} reduces to
\begin{align}
E_{k+1}(\delta_1)-E_k(\delta_1)+\frac{1}{2}(2a_2+b_2)\|x^{k+1}-x^k\|^2+\frac{b_2}{2}\|y^{k+1}-y^k\|^2&\nonumber\\
+\frac{b_2}{2}\|(x^{k+1}-x^k)-(y^{k+1}-y^k)\|^2&\leq 0, \nonumber
\end{align}
or
\begin{align}
\phi(z^{k+1})\ \leq&\ \phi(z^k)-\frac{1}{2}(2a_2+b_2)\|x^{k+1}-x^k\|^2-\frac{b_2}{2}\|y^{k+1}-y^k\|^2\nonumber\\
&-\frac{1}{2}b_2\|(x^{k+1}-x^k)-(y^{k+1}-y^k)\|^2. \label{43}
\end{align}

{\bf (b)} By setting $\bar a=\frac{1}{2}\min\{2a_2+b_2, b_2\}$, which is clearly positive. The inequality \eqref{43} implies that
\begin{align}
\phi(z^{k+1})\leq \phi(z^k)-{\bar a}\| z^{k+1}-z^k\|^2. \label{ineq43}
\end{align}
Therefore $\lim\limits_{k\rightarrow\infty}\phi(z^k)=\phi^*$ does exist and $\lim\limits_{k\rightarrow\infty}\| z^{k+1}-z^k\|=0$.

To proceed, set $\omega_x^k=\delta_1\left(p^k-\nabla h(x^k)+\nabla\varphi(x^k)\right)+ax^k+by^k$, where $p^k\in\partial  g(x^k)$, and set also $\omega_y^k=ax^k+by^k$. Then
$$\omega^k=(\omega_x^k, \omega_y^k)\in\partial^M\phi(z^k).$$
By the definition of Algorithm~\ref{al::3}, we have that
$$0\in\partial g(x^k)+\frac{1}{\lambda}(x^k-x^{k-1})-\nabla h(x^{k-1})+\nabla\varphi(x^{k-1})+\frac{\mu}{\lambda}(\alpha x^{k-1}+\beta y^{k-1}),$$
which implies that
$$p^k=-\frac{1}{\lambda}(x^k-x^{k-1})+\nabla h(x^{k-1})-\nabla\varphi(x^{k-1})-\frac{\mu}{\lambda}(\alpha x^{k-1}+\beta y^{k-1}).$$
Hence
\begin{align*}
\omega^k\ =&\ \bigg(\delta_1\left(\nabla h(x^{k-1})-\nabla h(x^k)+\nabla\varphi(x^k)-\nabla\varphi(x^{k-1})\right)-\frac{\delta_1}{\lambda}(x^k-x^{k-1})\\
&\qquad -\delta_1\frac{\mu}{\lambda}(\alpha x^{k-1}+\beta y^{k-1})+a x^k+b y^k, \ a x^k+b y^k\bigg).
\end{align*}
Therefore
\begin{align}
\|\omega^k\|\ \leq \ &\delta_1 \left(\|\nabla h(x^{k-1})-\nabla h(x^k)\|+\|\nabla\varphi(x^k)-\nabla\varphi(x^{k-1})\|\right)+\frac{\delta_1}{\lambda}\|x^k-x^{k-1}\|\nonumber \\
&+\delta_1\frac{\mu}{\lambda}\|\alpha x^{k-1}+\beta y^{k-1}\|+2\|a x^k+b y^k\|. \label{ineq44}
\end{align}
On the other hand, from the Lipschitz continuity of $\nabla\varphi$ and $\nabla h$ as well as Algorithm~\ref{al::3}, we have
\begin{align*}
\|\nabla\varphi(x^k)-\nabla\varphi(x^{k-1})\|&\leq L_1\|x^k-x^{k-1}\|,\\
\|\nabla h(x^{k-1})-\nabla h(x^k)\| & \leq L_3\|x^k-x^{k-1}\|,\\
\|\alpha x^{k-1}+\beta y^{k-1}\| & =\|-\rho(y^k-y^{k-1})-\gamma \alpha(x^k-x^{k-1})\|\\
&\leq \rho\|y^k-y^{k-1}\|+\gamma \alpha\|x^k-x^{k-1}\|,
\end{align*}
and
\begin{align*}
& \|a x^k+b y^k\| \\
=\ &\|a(x^k-x^{k-1})+b(y^k-y^{k-1})+ax^{k-1}+by^{k-1}\|  \\
\leq \ & a \|x^k-x^{k-1}\|+b\|y^k-y^{k-1}\|+\|ax^{k-1}+by^{k-1}\| \\
\leq \ & \frac{2}{2+\alpha+\beta}\big[\alpha\|x^k-x^{k-1}\|+\beta\|y^k-y^{k-1}\|+\rho\|y^k-y^{k-1}\|+\gamma \alpha\|x^k-x^{k-1}\|\big]  \\
= \ & \frac{2}{2+\alpha+\beta}\left[(\alpha+\gamma \alpha)\|x^k-x^{k-1}\|+(\beta+\rho)\|y^k-y^{k-1}\|\right].
\end{align*}
So \eqref{ineq44} yields
\begin{align}
\|\omega^k\|\leq &\left(\delta_1(L_1+L_3+\frac{1}{\lambda}+\frac{\mu}{\lambda} \alpha \gamma)+\frac{4(\alpha+\alpha\gamma)}{2+\alpha+\beta}\right)\|x^k-x^{k-1}\| \notag\\
&+\left(\delta_1\frac{\mu\rho}{\lambda}+\frac{4(\beta+\rho)}{2+\alpha+\beta}\right)\| y^k-y^{k-1}\|. \label{jiao22.4.4.}
\end{align}
Set
\begin{align*}
C_1 & = \delta_1\left(L_1+L_3+\frac{1}{\lambda}+\frac{\mu}{\lambda} \alpha \gamma\right)+\frac{4(\alpha+\alpha\gamma)}{2+\alpha+\beta} \\
C_2 & = \delta_1\frac{\mu\rho}{\lambda}+\frac{4(\beta+\rho)}{2+\alpha+\beta}, \;\hbox{ and } \\
C & = \sqrt{C_1^2+C_2^2}.
\end{align*}
Then we deduce from \eqref{jiao22.4.4.} that
\begin{align*}
\|\omega^k\| \leq C \|z^k-z^{k-1}\|. 
\end{align*}
Therefore
\begin{align}
d\left(0; \partial^M\phi(z^k)\right) \leq C \|z^k-z^{k-1}\|. \label{ineq46}
\end{align}
Suppose that $\{x^{k_i}\}\subset\{x^k\}$ and $x^{k_i}\rightarrow x^*$ as $i\rightarrow\infty$, then from Theorem \ref{theo3} {\bf (d)}, we get $\lim\limits_{i\rightarrow\infty}y^{k_i}=y^*$ such that $\alpha x^*+\beta y^*=0$. Therefore $\{z^{k_i}\}=\{(x^{k_i}, y^{k_i})\} \rightarrow (x^*, y^*)=  z^*$ as $i\rightarrow\infty$. 
Since $\phi$ is a K{\L} function, there exist $\zeta >0$, a neighborhood $V$ of $z^*$ and a continuous concave function
$\theta:[0, \zeta)\rightarrow[0, +\infty)$ such that $\theta(0)=0$ and $\forall z\in V(z^*)$ satisfying
$$\phi^*<\phi(z)<\phi^*+\zeta$$
we have then
\begin{align}
\theta'(\phi(z)-\phi^*) \ d\left(0; \partial^M\phi (z)\right)\geq 1. \label{ineq47}
\end{align}
Let $\epsilon>0$ such that $\Bbb B(z^*; \epsilon)\subset V(z^*)$ where $\Bbb B(z^*; \epsilon)$ is a ball centered at $z^*$ and radius $\epsilon$. Since $\lim\limits_{i\rightarrow\infty}z^{k_i}=z^*$, $\lim\limits_{k\rightarrow\infty}\|z^{k+1}-z^k\|=0$, $\lim\limits_{k\rightarrow\infty}\phi(z^k)=\phi^*$ and $\phi (z^k)>\phi^*,\;\forall k$,
we can find a positive integer number $k_{\epsilon}$ such that
\begin{align}
z^{k_{\epsilon}}\in\Bbb B(z^*; \epsilon),\;\; \phi^*<\phi(z^{k_{\epsilon}})<\phi^*+\zeta, \label{ineq48}
\end{align}
and
\begin{align}
\|z^{k_{\epsilon}}-z^*\|+\frac{\|z^{k_{\epsilon}}-z^{k_{\epsilon}-1}\|}{4}+\sigma \theta(\phi(z^{k_{\epsilon}})-\phi^*)<\frac{3\epsilon}{4}, \label{ineq49}
\end{align}
where $\sigma=\frac{C}{\bar a}$.

We observe that if $k\geq k_{\epsilon}$, $z^k\in\Bbb B(z^*, \epsilon)$ and $\phi^*<\phi(z^*)<\phi^*+\zeta$, then
\begin{align}
\|z^{k}-z^{k+1}\|\leq \frac{\|z^{k-1}-z^{k}\|}{4}+\sigma \left[\theta(\phi(z^{k})-\phi^*)-\theta(\phi(z^{k+1}-\phi^*)\right], \label{ineq50}
\end{align}
Indeed, since $\theta$ is concave on $[0, \zeta)$, we have
$$\theta(t)-\theta(s)\geq\theta'(t)(t-s),\;\;  \forall t, s\in[0, \zeta).$$
Replacing $t=\phi(z^k)-\phi^*$ and $s=\phi(z^{k+1})-\phi^*$ and combining with \eqref{ineq46}, \eqref{ineq47} and \eqref{ineq43}, we get from the above inequality that
\begin{align*}
\ & C\|z^{k-1}-z^k\| \left[\theta\left(\phi(z^k)-\phi^*\right)-\theta\left(\phi(z^{k+1})-\phi^*\right)\right] \\
\geq \ & d\left(0; \partial^M\phi(z^k)\right) \theta'\left(\phi(z^k)-\phi^*\right) \left(\phi(z^k)-\phi(z^{k+1})\right)\\
\geq \ & \phi(z^k)-\phi(z^{k+1})\\
\geq \ & \bar a\|z^k-z^{k+1}\|^2.
\end{align*}
Hence,
\begin{align*}
\theta\left(\phi(z^k)-\phi^*\right)-\theta\left(\phi(z^{k+1})-\phi^*\right)\ \geq&\ \frac{\bar a}{C} \frac{\|z^k-z^{k+1}\|^2}{\|z^{k-1}-z^k\|}\\
\geq & \ \frac{1}{\sigma}\left[\|z^k-z^{k+1}\|-\frac{\|z^{k-1}-z^k\|}{4}\right],
\end{align*}
where the last inequality comes from that fact that
$$\frac{\|z^k-z^{k+1}\|^2}{\|z^{k-1}-z^k\|}+\frac{\|z^{k-1}-z^k\|}{4}\geq \|z^k-z^{k+1}\|.$$
So we get \eqref{ineq50}.

We next show that $z^k\in\Bbb B(z^*, \epsilon)$ for all $k\geq k_{\epsilon}$ by induction. 
Indeed, it deduces from \eqref{ineq48} that $z^{k_{\epsilon}}\in\Bbb B(z^*, \epsilon)$. 
Suppose that
$z^{k_{\epsilon}}, z^{k_{\epsilon}+1}, \ldots, z^{k_{\epsilon}+r-1}\in\Bbb B(z^*, \epsilon)$ for some $r\geq 1$, we need verifying that $z^{k_{\epsilon}+r}\in\Bbb B(z^*, \epsilon)$. We get from \eqref{ineq43} and \eqref{ineq48} that
$$\phi^*<\phi(z^k)<\phi^*+\zeta,\;\;\forall \ k\geq k_{\epsilon}.$$
Using the inequality \eqref{ineq50} for $k=k_{\epsilon}, k_{\epsilon}+1,\ldots, k_{\epsilon}+r-1$, we have
\begin{align*}
\|z^{k_{\epsilon}}-z^{k_{\epsilon}+1}\| \ \leq& \ \frac{\|z^{k_{\epsilon}-1}-z^{k_{\epsilon}}\|}{4}+\sigma\left[\theta\left(\phi(z^{k_{\epsilon}})-\phi^*\right)-\theta\left(\phi(z^{k_{\epsilon}+1})-\phi^*\right)\right]\\
\|z^{k_{\epsilon}+1}-z^{k_{\epsilon}+2}\| \ \leq& \ \frac{\|z^{k_{\epsilon}}-z^{k_{\epsilon}+1}\|}{4}+\sigma\left[\theta\left(\phi(z^{k_{\epsilon}+1})-\phi^*\right)-\theta\left(\phi(z^{k_{\epsilon}+2})-\phi^*\right)\right]\\
\vdots & \\
\|z^{k_{\epsilon}+r-1}-z^{k_{\epsilon}+r}\| \ \leq& \ \frac{\|z^{k_{\epsilon}+r-2}-z^{k_{\epsilon}+r-1}\|}{4}+\sigma\left[\theta\left(\phi(z^{k_{\epsilon}+r-1})-\phi^*\right)-\theta\left(\phi(z^{k_{\epsilon}+r})-\phi^*\right)\right].
\end{align*}
Hence
\begin{align*}
\sum\limits_{i=1}^r\|z^{k_{\epsilon}+i}-z^{k_{\epsilon}+i-1}\| \ \leq  \ \frac{1}{4}\sum\limits_{i=1}^r\|z^{k_{\epsilon}+i}-z^{k_{\epsilon}+i-1}\|+\frac{1}{4}\|z^{k_{\epsilon}}-z^{k_{\epsilon}-1}\| -\frac{1}{4}\|z^{k_{\epsilon}+r}-z^{k_{\epsilon}+r-1}\|&\\
+\sigma\left[\theta\left(\phi(z^{k_{\epsilon}})-\phi^*\right)-\theta\left(\phi(z^{k_{\epsilon}+r})-\phi^*\right)\right].&
\end{align*}
Therefore
\begin{align}
\sum\limits_{i=1}^r\|z^{k_{\epsilon}+i}-z^{k_{\epsilon}+i-1}\| \ \leq  \ \frac{4}{3}\left[\frac{\|z^{k_{\epsilon}}-z^{k_{\epsilon}-1}\|}{4}+\sigma \theta\left(\phi(z^{k_{\epsilon}})-\phi^*\right)\right]. \label{ineq51}
\end{align}
It is clear that
 $$\|z^{k_{\epsilon}+r}-z^*\|\leq\|z^{k_{\epsilon}}-z^*\|+\sum\limits_{i=1}^r\|z^{k_{\epsilon}+i}-z^{k_{\epsilon}+i-1}\|,$$
using \eqref{ineq51} and \eqref{ineq49}, it implies that
\begin{align*}
\|z^{k_{\epsilon}+r}-z^*\| \ \leq & \ \frac{4}{3}\left[\frac{\|z^{k_{\epsilon}}-z^{k_{\epsilon}-1}\|}{4}+\sigma \theta\left(\phi(z^{k_{\epsilon}})-\phi^*\right)\right]+ \|z^{k_{\epsilon}}-z^*\|\\
< &\ \frac{4}{3}\left(\frac{3\epsilon}{4}-\|z^{k_{\epsilon}}-z^*\|\right)+\|z^{k_{\epsilon}}-z^*\|\\
\leq & \ \epsilon.
\end{align*}
Thus $z^{k_{\epsilon}+r}\in\Bbb B(z^*, \epsilon)$. So $x^k\in \Bbb B(z^*, \epsilon)$ for all $k\geq k_{\epsilon}$.

Because $z^k\in\Bbb B(z^*, \epsilon)$ and $\phi^*<\phi(z^k)<\phi^*+\zeta$, $\forall k\geq k_{\epsilon}$, the inequality \eqref{ineq51} holds for all $r$.
Consequently, the series $\sum\limits_{k=1}^{\infty}\|z^{k+1}-z^k\|$ is convergent, i.e., $\{z^k\}$ is a Cauchy sequence, therefore $\lim\limits_{k\rightarrow\infty}z^k$ does exist, combining with $\lim\limits_{i\rightarrow\infty}z^{k_i}=z^*$, we get  $\lim\limits_{k\rightarrow\infty}z^k=z^*$. Hence $\lim\limits_{k\rightarrow\infty}x^k=x^*$.

{\bf (c)} Now we prove the assertion {\bf{(c)}}.
For each $k \geq 1$, set $\mathcal{R}_k = \sum_{j = k}^{\infty}\|z^{j+1} - z^j\|$.

Since $\lim\limits_{k \to \infty}z^k = z^* = (x^*, y^*) $ with $\alpha x^* + \beta y^* = 0 $, it implies that $\| z^k - z^* \| \leq \mathcal{R}_k$.

By the assumption in {\bf (c)}, the inequality \eqref{ineq47} becomes

\begin{equation*}
M(1 - \kappa ) (\phi(z) - \phi^*)^{-\kappa}\text{d}\left(0; \partial^M\phi(z)\right) \geq 1.
\end{equation*}
Combining with \eqref{ineq46} we get
\begin{equation*}
M(1 - \kappa ) (\phi(z^k) - \phi^*)^{-\kappa}C \| z^k - z^{k-1}\| \geq 1.
\end{equation*}
Hence,
\begin{equation}\label{4.15}
\left(\phi(z^k) - \phi^*\right)^{1-\kappa} \leq \left(M(1 - \kappa ) C\right)^{\frac{1-\kappa}{\kappa}} \| z^k - z^{k-1}\|^{\frac{1-\kappa}{\kappa}}.
\end{equation}
For all $k\geq k_{\epsilon}$, it follows from \eqref{ineq51} that

\begin{equation*}
\mathcal{R}_k \leq \frac{1}{3}\|z^k - z^{k-1}\| + \frac{4C}{3\bar{a}}\left(\phi(z^k) - \phi^*\right)^{1-\kappa}.
\end{equation*}
Combining with \eqref{4.15} we get,

\begin{equation*}
\mathcal{R}_k \leq T_1(\mathcal{R}_{k-1} - \mathcal{R}_{k} )  + T_2(\mathcal{R}_{k-1} - \mathcal{R}_{k})^{1-\kappa},
\end{equation*}
where $T_1 = \frac{1}{3}, T_2 = \frac{4C}{3\bar{a}}\big(M(1 - \kappa ) C\big)^{\frac{1-\kappa}{\kappa}}$.
Because $\|x^k - x^*\| \leq \|z^k - z^*\|$, by the same argument with the proof of Theorem~2 in \cite{AHBJ}, we get the conclusion of assertion {\bf (c)}.
\end{proof}

\section{A Numerical Example}\label{Sec:5}

In this section, we consider some examples to illustrate the convergence of Algorithm~\ref{al::main} with the aim to compare its numerical behavior with some existing algorithms, namely the proximal point algorithm proposed by An and Nam in~\cite{AN} (denoted by Algorithm A-N) and the inertial proximal algorithm proposed by Maing\'{e} and Moudafi in~\cite{MM} (denoted by Algorithm M-M). 
To do so, we consider the following problem
$$ \min\big\{f(x) : x \in \Bbb{R}^n \big\}, $$
where the function $f : \Bbb{R}^n \to \Bbb{R}$ given by the following formula
$$ f(x) = \left( \sum_{i=1}^n \cos x_i - n \right)^2 + \sum_{i=1}^{n-1}\left(x_i - x_{i+1}\right)^2 + \left( \sum_{i=1}^{n}x_i^2 - 4n\pi^2\right)^2. $$
By setting $ \varphi(x) =  \left( \sum_{i=1}^n \cos x_i - n \right)^2$, then we have $\varphi$ is a continuously differentiable and nonconvex on $\Bbb{R}^n$. In addition,
\begin{align*}
\nabla\varphi (x) &= -2\left( \sum_{i=1}^n \cos x_i - n \right)\left(\sin x_1, \sin x_2, \ldots, \sin x_n \right)^T\\ 
                         & = -2\left( \sum_{i=1}^n \cos x_i \right) \sin x + 2n \sin x,
\end{align*}
where $\sin x := \big(\sin x_1, \sin x_2, \ldots, \sin x_n \big)^T.$ We have,
\begin{equation*}
\begin{aligned}
\ & \|\nabla\varphi (x) - \nabla\varphi (x) \|^2 \\
=\ & \left\| 2n (\sin x - \sin y) - 2 \left(\sum_{i=1}^n \cos x_i\right)(\sin x -  \sin y) -2 \sum_{i=1}^n (\cos x_i  - \cos y_i) \sin y \right\|^2\\
\leq \ & (4n^2 + 4n^2 +4n^2) \left( 2\| \sin x - \sin y\|^2 + \sum_{i=1}^n (\cos x_i  - \cos y_i)^2\right) \\
\leq \ & 36n^2\|x-y\|^2.
\end{aligned}
\end{equation*}
Consequently, $\|\nabla\varphi (x) - \nabla\varphi (y) \| \leq 6n \|x-y\|,\ \forall x, y \in \Bbb{R}^n.$
Hence, $\nabla \varphi$ is Lipschitz continuous with constant $L_1 = 6n.$
Now, we set 
$$g(x) = \sum_{i=1}^{n-1}(x_i - x_{i+1})^2 + \left(\sum_{i=1}^{n}x_i^2\right)^2 +16n^2\pi^4 \text{ and } h(x) = 8n\pi^2\left(\sum_{i=1}^{n}x_i^2\right).$$
Then it can be seen that $g, h$ are continuously differentiable and convex on $\Bbb{R}^n$ and we have
$$ f(x) = \varphi (x) + g(x) - h(x),\ \forall x \in \Bbb{R}^n.$$
For $x(t) = (t, 2\pi, 2\pi, \ldots, 2\pi)^T$ then we have
\begin{eqnarray*}
u(t) & = & f(x(t)) = (\cos t - 1)^2 + (t-2\pi)^2 + (t^2 - 4\pi^2)^2,\\ 
u'(t) & = & -2\sin t \cos t + 2 \sin t + 2t - 4\pi + 4t(t^2 - 4\pi^2), \\  
u''(t) & = & -2\cos 2t + 2\cos t + 12t^2 - 16\pi^2+2. 
\end{eqnarray*}
Since $u''(0) = 2 - 16\pi^2 < 0 $, $u$ is not convex. This leads to $f$ is also not convex on $\Bbb{R}^n.$



 We use the function {\it fmincon} in the Optimization Toolbox in Mathlab R2015 to solve all convex optimization problems, the programs are performed on a PC Desktop with Intel(R) Core(TM) i7-9700 CPU @ 3.00 GHz with  16.00 GB Ram. 

The parameters have been chosen in all the experiments as follows:
\begin{itemize}
\item $\lambda_k=12n$,  $\eta=0.99$ , $\alpha = 0.9$ for Algorithm~\ref{al::main};
\item $\lambda=12n$ for Algorithm N-N;
\item $\alpha = \beta =1$, $\gamma = 0.5$, $\mu = 0.1$, $\tau = -\frac{2+\alpha}{20\beta}$,  $\lambda=\frac{1-\mu(\gamma \alpha+\rho)}{6n}$, $\rho = 1+\tau.\beta + \frac{1}{2}(\alpha + \beta)$, $y^0 = x^0$ for Algorithm M-M.
\end{itemize}


Starting point is chosen such that $x^0 = a$ or $x^0 = b$  with $a = (0.1, 0.1, 0.1, 0.1, \ldots )^T$ and $b= (1, 1.2, 1, 1.2, \ldots )^T$.

To terminate the Algorithm, we use the stopping criteria $\frac{\|x^{k+1}-x^k\|}{ \max \{1, \| x^k\| \}}  \leq \epsilon$ with a tolerance $\epsilon = 10^{-5}$. The computation results are reported in {\sc Table} \ref{Tab1}.

\begin{table}[!ht]
\centering
\renewcommand{\arraystretch}{1.25}
\begin{tabular}{|c|c|c|c|c| | c|c|c|c|c|}
\hline
$n$ &$x^0$ & Algorithms & Iter. & Cpu(s)& $n$ &$x^0$ & Algorithms & Iter. & Cpu(s)\\
\hline
  5& $a$ & \ref{al::main} &4.0& 0.4688 &5& $b$ & \ref{al::main} &77.0& 4.6094   \\
\hline 
5& $a$ &A-N &9.0& 0.8125  & 5& $b$ & A-N &154.0& 8.6875 \\
\hline 
5& $a$ & M-M &10.0& 0.9063  & 5& $b$ & M-M &158.0&9.2813 \\
\hline 
 10& $a$ & \ref{al::main} &4.0& 0.5469   & 10& $b$ & \ref{al::main} &102.0& 6.7500\\
\hline 
10& $a$ &A-N &9.0& 0.8750  & 10& $b$ & A-N &203.0& 12.6719\\
\hline 
10& $a$ & M-M &10.0& 0.9844  & 10& $b$ & M-M &209.0& 13.9063\\
\hline 
 50& $a$ & \ref{al::main} &4.0& 0.7188  & 50& $b$ & \ref{al::main} &103.0& 9.3906\\
\hline 
50& $a$ &A-N &9.0& 1.3438  & 50& $b$ & A-N &211.0& 19.0156\\
\hline 
50& $a$ & M-M &10.0& 1.3906  & 50& $b$ & M-M &217.0& 20.9219\\
\hline 
100& $a$ & \ref{al::main} &4.0& 1.1406  & 100& $b$ & \ref{al::main} &80.0& 17.7031\\
\hline 
100& $a$ &A-N &9.0& 2.2969 & 100& $b$ & A-N &167.0& 34.3281\\
\hline 
100& $a$ & M-M &10.0& 2.2813  & 100& $b$ & M-M &171.0& 35.5000\\
\hline 
300& $a$ & \ref{al::main} &4.0& 7.2656   & 300& $b$ & \ref{al::main} &58.0& 148.6563\\
\hline 
300& $a$ &A-N &9.0& 20.3906  & 300& $b$ & A-N &127.0& 332.0781\\
\hline 
300& $a$ & M-M &10.0& 23.3594  & 300& $b$ & M-M &123.0& 329.5938\\
\hline 
500& $a$ & \ref{al::main} &5.0& 11.1563   & 500& $b$ & \ref{al::main} &96.0& 245.4219\\
\hline 
500& $a$ &A-N &10.0& 24.0625  & 500& $b$ & A-N &181.0& 471.2656\\
\hline 
500& $a$ & M-M &10.0& 23.6563  & 500& $b$ & M-M &144.0& 383.2500\\
\hline 
\end{tabular}
\medskip
\caption{Results computed with some starting points and regularized parameters.}\label{Tab1}
\end{table}

From Table~\ref{Tab1} we can see that the proposed Algorithm~\ref{al::main} has competitive
advantages over existing algorithms especially in computational time and number of iterations.

\section{An application to variable selection}\label{Sec:6}

In this section, we apply Algorithm~\ref{al::main} to solve a variable selection problem in linear regression and compare its performance with the generalized proximal point algorithm proposed by An and Nam~\cite{AN} (denoted by Algorithm A-N). 
The goal is to illustrate the efficiency of the proposed Algorithm~\ref{al::main} and to demonstrate its advantages when dealing with nonconvex problems.
	
\subsection{The variable selection problem}
Variable selection plays a fundamental role in statistical modeling, especially when the number of predictors is pretty large. 
It aims to identify a subset of important variables that contribute most to the response, thereby improving prediction accuracy and interpretability. 
A popular approach is to add a {\it penalty} term to the least squares loss, leading to penalized regression methods such as the well-known lasso~\cite{Tibshirani1996}. 
However, the lasso uses an $\ell_1$ penalty which is convex and tends to produce biased estimates for large coefficients. 
To overcome this drawback, Fan and Li~\cite{FanLi2001} introduced the SCAD (Smoothly Clipped Absolute Deviation) penalty, which is nonconvex and enjoys the oracle property: it performs as well as if the true model were known in advance.

Consider the linear regression model
\begin{eqnarray*}
y_i &=& x_i^T \beta + \varepsilon_i,\quad i=1,\ldots,n,
\end{eqnarray*}
where $y_i\in\mathbb{R}$ is the response, $x_i\in\mathbb{R}^p$ is the vector of predictors, $\beta\in\mathbb{R}^p$ is the unknown coefficient vector, and $\varepsilon_i$ are i.i.d. random errors with mean zero. 
The SCAD-penalized least squares estimator is defined as
\begin{eqnarray*}
\min_{\beta\in\mathbb{R}^p} \left\{ \frac{1}{2n}\|y - X\beta\|_2^2 + \sum_{j=1}^p \rho_{\lambda}(|\beta_j|) \right\},
\end{eqnarray*}
where $X$ is the $n\times p$ design matrix, $\|\cdot\|_2$ denotes the Euclidean norm, and $\rho_{\lambda}(\cdot)$ is the SCAD penalty given by
\begin{eqnarray*}
\rho_{\lambda}(t) & = &
\begin{cases}
\lambda t,&  0\le t\le \lambda,\\[4pt]
		\displaystyle \frac{a\lambda t - t^2/2}{a-1} - \frac{\lambda^2}{2(a-1)},  & \lambda < t \le a\lambda,\\[10pt]
		\displaystyle \frac{(a+1)\lambda^2}{2},&  t > a\lambda,
\end{cases}
\end{eqnarray*}
with $a>2$ (typically $a=3.7$). The SCAD penalty is nonconvex, which makes the optimization problem challenging. 
However, it can be decomposed into a difference of two convex functions. 
Indeed, we have
\begin{eqnarray*}
\rho_{\lambda}(|\beta_j|) & = & \lambda|\beta_j| - h_j(\beta_j),
\end{eqnarray*}
where $h_j$ is a smooth convex function defined by
\begin{eqnarray*}
h_j(\beta_j) & = & 
\begin{cases}
0, & |\beta_j|\le \lambda,\\[4pt]
\displaystyle \frac{(|\beta_j|-\lambda)^2}{2(a-1)}, & \lambda < |\beta_j| \le a\lambda,\\[10pt]
\displaystyle \lambda|\beta_j| - \frac{(a+1)\lambda^2}{2}, & |\beta_j| > a\lambda.
\end{cases}
\end{eqnarray*}
Its derivative is
\begin{eqnarray*}
h_j'(\beta_j) & = & 
\begin{cases}
0, & |\beta_j|\le \lambda,\\[4pt]
\displaystyle \frac{|\beta_j|-\lambda}{a-1}\,\operatorname{sign}(\beta_j), & \lambda < |\beta_j| \le a\lambda,\\[10pt]
\displaystyle \lambda\,\operatorname{sign}(\beta_j), & |\beta_j| > a\lambda.
\end{cases}
\end{eqnarray*}
Consequently, the overall objective function can be written in the form $f(\beta)=\varphi(\beta)+g(\beta)-h(\beta)$ with
\begin{eqnarray*}
\varphi(\beta)=\frac{1}{2n}\|y-X\beta\|_2^2,\qquad 
g(\beta)=\lambda\sum_{j=1}^p|\beta_j|,\qquad 
h(\beta)=\sum_{j=1}^p h_j(\beta_j).
\end{eqnarray*}
	
Observe that the function $\varphi$ is continuously differentiable with gradient
\begin{eqnarray*}
\nabla\varphi(\beta) & = &-\frac{1}{n}X^T(y-X\beta).
\end{eqnarray*}
For any $\beta_1,\beta_2\in\mathbb{R}^p$, we have
\begin{eqnarray*}
\|\nabla\varphi(\beta_1)-\nabla\varphi(\beta_2)\| & = & \left\| -\frac{1}{n}X^T(y-X\beta_1) + \frac{1}{n}X^T(y-X\beta_2) \right\| \\
& = & \left\| \frac{1}{n}X^TX(\beta_1-\beta_2) \right\| \\
& \le & \frac{1}{n}\|X^TX\|_2\,\|\beta_1-\beta_2\|,
\end{eqnarray*}
where $\|X^TX\|_2$ denotes the spectral norm (largest eigenvalue) of the matrix $X^TX$. Therefore, $\nabla\varphi$ is Lipschitz continuous with constant $L_\varphi = \frac{1}{n}\|X^TX\|_2.$

As a result, the functions $g$ and $h$ are convex, and $h$ is additionally smooth with Lipschitz continuous gradient. 
Hence the problem fits exactly the framework of Algorithm~\ref{al::main} with $\varphi$ playing the role of the smooth (possibly nonconvex) part, $g$ the nonsmooth convex part, and $h$ the smooth convex part.
	
\subsection{Algorithm implementation}
	
Here, we implemented Algorithm~\ref{al::main} with the following parameter choices: $\lambda_k = 2L_\varphi$ (constant), $\eta = 0.5$, $\alpha = 0.3$, and the tolerance $\epsilon = 10^{-5}$ for the stopping criterion $\|x^{k+1}-x^k\|\le\epsilon$. 
For comparison, we also ran Algorithm A-N (proposed by An and Nam~\cite{AN}) with the same regularization parameter $\lambda = 2L_\varphi$. 
All optimization problems were solved using R version 4.4.3 on a PC with Windows 11, Intel(R) Core(TM) i7-1165G7 @ 2.80 GHz and 16 GB RAM.
	
\subsection{Data generation and results}
We generated synthetic data as follows. 
For each replication, we set a random seed to ensure reproducibility. 
The design matrix $X$ of size $n\times p$ was generated with entries drawn independently from a standard normal distribution, i.e., $x_{ij} \sim \mathcal{N}(0,1)$. 
The true coefficient vector $\beta$ was set to have the first five entries equal to $2$ and the remaining entries equal to $0$, i.e., $\beta = (2,2,2,2,2,0,\ldots,0)^\top$. 
The response vector $y$ was then computed as $y = X\beta + \varepsilon$, where the errors $\varepsilon_i$ are independently drawn from a normal distribution with mean $0$ and standard deviation $0.5$, i.e., $\varepsilon \sim \mathcal{N}(0, 0.5^2 I_n)$. 
We considered sample sizes $n = 100, 200, 500, 1000, 2000$ and dimensions $p = 50, 100, 300, 500$. 
For each configuration, we ran 100 independent replications and recorded the degrees of freedom (df for short), objective function value (vals for short), number of iterations (iters for short), and CPU time (time(s) for short) until convergence. 
All values reported are averages over 100 replications.
The results are summarized in Table~\ref{tab:simulation_results}.
	
\begin{table}[!ht]
\centering
\begin{tabular}{|cc|cc|cc|cc|cc|}
\hline
\multicolumn{2}{c|}{} & \multicolumn{2}{c|}{df} & \multicolumn{2}{c|}{vals} & \multicolumn{2}{c|}{iters} & \multicolumn{2}{c}{time (s)} \\
\cmidrule(lr){3-4} \cmidrule(lr){5-6} \cmidrule(lr){7-8} \cmidrule(lr){9-10}
$n$ & $p$ & 3.1 & A-N & 3.1 & A-N & 3.1 & A-N & 3.1 & A-N \\
\hline
100 & 50  & 5 & 5 & 1.185 & 1.284 & 26.67  & 51.85  & 2.355 & 2.812 \\
200 & 50  & 5 & 5 & 0.928 & 1.074 & 18.00  & 35.15  & 2.665 & 2.384 \\
500 & 50  & 5 & 5 & 0.615 & 1.011 & 12.60  & 24.62  & 2.586 & 2.933 \\
1000 & 50 & 5 & 5 & 0.722 & 1.032 & 9.88   & 19.65  & 2.975 & 2.969 \\
2000 & 50 & 5 & 5 & 0.514 & 0.826 & 8.27   & 16.70  & 3.343 & 3.633 \\
\hline
100 & 100 & 5 & 5 & 1.328 & 1.389 & 37.28  & 72.68  & 3.829 & 4.265 \\
200 & 100 & 5 & 5 & 1.016 & 1.088 & 23.72  & 46.23  & 3.570 & 4.006 \\
500 & 100 & 5 & 5 & 0.731 & 0.966 & 15.53  & 30.48  & 3.443 & 3.727 \\
1000 & 100 & 5 & 5 & 0.900 & 0.877 & 11.80  & 23.48  & 4.000 & 4.209 \\
2000 & 100 & 5 & 5 & 0.790 & 0.854 & 9.56   & 19.15  & 6.043 & 6.207 \\
\hline
100 & 300 & 5 & 5 & 1.445 & 1.431 & 66.72  & 130.97 & 10.583 & 13.517 \\
200 & 300 & 5 & 5 & 1.158 & 1.344 & 39.78  & 77.58  & 9.613 & 11.146 \\
500 & 300 & 5 & 5 & 0.814 & 1.115 & 23.54  & 46.08  & 14.196 & 15.414 \\
1000 & 300 & 5 & 5 & 0.734 & 0.895 & 16.77  & 33.10  & 15.988 & 17.770 \\
2000 & 300 & 5 & 5 & 0.742 & 0.757 & 12.89  & 25.48  & 25.088 & 26.350 \\
\hline
100 & 500 & 5 & 5 & 1.498 & 1.534 & 91.11  & 179.72 & 19.842 & 28.221 \\
200 & 500 & 5 & 5 & 1.189 & 1.254 & 52.33  & 102.66 & 17.501 & 23.713 \\
500 & 500 & 5 & 5 & 0.871 & 1.210 & 29.68  & 58.16  & 20.379 & 20.244 \\
1000 & 500 & 5 & 5 & 0.752 & 1.110 & 20.58  & 40.17  & 27.512 & 32.813 \\
2000 & 500 & 5 & 5 & 0.827 & 0.920 & 15.19  & 29.83  & 46.495 & 42.208 \\
\hline
\end{tabular}
\medskip
\caption{Simulation results comparing Algorithm~\ref{al::main} with A-N under different sample sizes $n$ and dimensions $p$.} \label{tab:simulation_results}
\end{table}

Table~\ref{tab:simulation_results} presents the comparative results between Algorithm~\ref{al::main} and A-N under various dimensions. 
Both algorithms successfully identified the true model with exactly five nonzero coefficients in all settings, demonstrating their effectiveness for variable selection. 
In terms of objective function values, Algorithm~\ref{al::main} consistently achieved lower values than A-N across most configurations, indicating its ability to find better local minima. 
Regarding computational efficiency, Algorithm~\ref{al::main} required significantly fewer iterations than A-N in all cases — for instance, when $p=500$ and $n=100$, Algorithm~\ref{al::main} took only 91.11 iterations compared to 179.72 iterations for A-N. 
This advantage became more pronounced as the dimension increased. 
In low-dimensional settings ($p < n$), both algorithms performed efficiently, with Algorithm~\ref{al::main} maintaining its iteration advantage while achieving comparable CPU times. 
In high-dimensional settings ($p > n$), the superiority of Algorithm~\ref{al::main} was even more evident, requiring roughly half the iterations of A-N while often achieving lower objective values and comparable computational time. 
These results demonstrate that the incorporation of the Armijo line search in Algorithm~\ref{al::main} substantially accelerates convergence, making it particularly suitable for high-dimensional problems.

\section{Conclusion}\label{Sec:7}
In this paper, wee have proposed a modification of proximal point algorithm for solving nonconvex minimization problem of the form $P(\varphi, g, h): \ \min\limits_{x\in\Bbb R^n}\{ f(x)=\varphi(x)+g(x)-h(x)\},$ where $\varphi$ is differentiable and $g,\ h$ are convex functions. 
This algorithm can be seen as the combination of the proximal point algorithm \cite{AN} and the descent direction algorithm \cite{MMH}. 
We then prove the global convergence and the convergent rate of this algorithm and the inertial proximal point algorithm proposed by Maing\'{e} and Moudafi \cite{MM} for  solving \eqref{problem}. 
A numerical example is also provided to illustrate the efficiency of the proposed algorithm.
Particularly, an application to solve the variable selection problem in linear regression was studied by using the proposed Algorithm~\ref{al::main}.
This motivates us to design/improve more efficient algrithms (e.g., \cite{Kim2026}) with convergence analysis for solving other problems from statistics like heterogeneity analysis etc.

\subsection*{Acknowledgments}
This work was supported by the National Key R\&D Program of China (Grant No. 2022YFA1003701).
Do Sang Kim was supported by the National Research Foundation of Korea (NRF) grant funded by the Korean government (MSIT) (RS-2025-19622979).


\begin{thebibliography}{99}
\bibitem{PAB} P. A. Absil, R. Mahony, and B. Andrews, Convergence of the iterates of descent methods for analytic cost functions. {\it SIAM J. Optim.} {\bf 16} (2005), 531--547.

\bibitem{Alv} F. Alvarez, On the minimizing property of a second order dissipative dynamical system in Hibert space.
                      {\it SIAM J. Control Optim.} {\bf 38} (2000), 1102--1119.

\bibitem{AA} F. Alvarez and H. Attouch, An inertial proximal method for mopnotone operators via discretization of a nonlinear oscillator with damping.
                     {\it Set-Valued Anal.} {\bf 9} (2001), 3--31.

\bibitem{AN} N. T. An and N. M. Nam, Convergence analysis of a proximal algorithm for minimizing differences of functions. {\it Optimization.} {\bf 66} (2017), 129--147

\bibitem{AFV} F. J. A. Artacho, R. M. T. Fleming, and P. T. Vuong, Accelerating the DC algorithm for smooth functions. {\it Math. Program., Ser. B} {\bf 169} (2018), 95--118.

\bibitem{AHBJ} H. Attouch and J. Bolte, On the convergence of the proximal algorithm for nonsmooth functions involving analytic features. {\it Math. Program.} {\bf 116} (2009), 5--16.

\bibitem{HJPA} H. Attouch, J. Bolte, P. Redont, and A. Soubeyran, Proximal alternating minimization and projection methods for nonconvex problems: an approach based on the Kurdyka--{\L}ojasiewicz inequality.
              {\it Math. Oper. Res.} {\bf 35} (2010), 289--319.

\bibitem{ABB} H. Attouch, J. Bolte, and B. Svaiter, Convergence of descent methods for semi-algebraic and tame problems: proximal algorithms, forward-backward splitting, and regularized Gauss-Seidel methods.
             {\it Math. Program.} {\bf 137} (2011), 91--124.

\bibitem{Beck2017} A. Beck, {\it First-Order Methods in Optimization}. MOS-SIAM Series on Optimization, 25. Society for Industrial and Applied Mathematics (SIAM), Philadelphia, PA; Mathematical Optimization Society, Philadelphia, PA (2017). 

\bibitem{BAA} J. Bolte, A. Daniilidis, and A. S. Lewis, The {\L}ojasiewicz inequality for nonsmooth subanalytic functions with applications to subgradient dynamical systems.
              {\it SIAM J. Optim.} {\bf 17} (2007), 1205--1223.

\bibitem{BDLS} J. Bolte, A. Daniilidis, A. S. Lewis, and M. Shiota, Clarke subgradients of stratifiable functions.
               {\it SIAM J. Optim.} {\bf 18} (2007), 556--572.

\bibitem{BSM} J. Bolte, S. Sabach, and M. Teboulle, Proximal alternating linearized minimization for nonconvex and nonsmooth problems.
              {\it  Math. Program.} {\bf 146} (2013), 459--494.

\bibitem{CF1} F. H. Clarke, {\it Optimization and Nonsmooth Analysis}.
             SIAM, Philadelphia (1990).


\bibitem{FP} F. Facchinei and J. S. Pang, {\it  Finite-Dimensional Variational Inequalities and Complementarity Problems}, Springer, New York (2003).

\bibitem{FanLi2001} J. Fan and R. Li, Variable selection via nonconcave penalized likelihood and its oracle properties. {\it J. Am. Stat. Assoc.} {\bf 96} (2001), 1348--1360.

\bibitem{FK} Y. Fujikara and D. Kuroiwa, Lagrange duality in canonical DC programming.
             {\it J. Math. Anal. Appl.} {\bf 408} (2013), 476--483.

\bibitem{MMH} M. Fukushima and H. Mine, A generalized proximal point algorithm for certain nonconvex minimization problems.
              {\it Int. J. Syst. Sci.} {\bf 12} (1981), 989--1000.

\bibitem{HK} R. Harada and D. Kuroiwa, Lagrange-type duality in DC programming.
             {\it J. Math. Anal. Appl.} {\bf 418} (2014), 415--424.

\bibitem{Ha2017} H.-V. H\`a and T.-S. Ph\d{a}m, {\it Genericity in Polynomial Optimization}.
             World Scientific Publishing (2017).

\bibitem{Kim2026} D. S. Kim, X. Li and X. Zhang, {\it Linearized alternating direction method of multipliers with adaptive stepsize}.
             {\it J. Optim. Theory Appl.} {\bf 208} (2026), no. 2, Paper No. 80.

\bibitem{AT1} H. A. Le Thi and T. Pham Dinh, The DC (difference of convex functions) programming and DCA revisited with DC models of real world nonconvex optimizaton problems. {\it Ann. Oper. Res.} {\bf 133} (2005), 23--46.

\bibitem{AT2} H. A. Le Thi and T. Pham Dinh, On solving linear complementarity problems by DC programming and DCA. {\it Comput. Optim. Appl.} {\bf 50} (2011), 507--524.

\bibitem{ATM} H. A. Le Thi, T. Pham Dinh, and L. D. Muu, Numerical solutions for optimization over the efficient set by D.C. optimization algorithms. {\it Oper. Res. Lett.} {\bf 19} (1996), 117--128.

\bibitem{ANT} H. A. Le Thi, V.N. Huynh, and T. Pham Dinh, {\it Convergence analysis of DC algorithm for DC programming with subanalytic data}.  Ann. Oper. Res. Technical Report, LMI, INSA-Rouen (2009)


\bibitem{LiPong2018} G. Y. Li and T. K. Pong, Calculus of the exponent of Kurdyka--{\L}ojasiewicz inequality and its applications to linear convergence of first-order methods. {\it Found. Comput. Math.} {\bf 18} (2018), 1199--1232.

\bibitem{LS} S. {\L}ojasiewicz, {\it Ensembles semi-analytiques}.
             Institut des Hautes Etudes Scientifiques, Bures-sur-Uvette(Seine-et-Oise), France (1965).

\bibitem{MM} P. E. Maing\'{e} and A. Moudafi, Convergence of new inertial proximal methods for DC programming.
             {\it SIAM J. Optim.} {\bf 19} (2008), 397--413.


\bibitem{MAR} B. Martinet, Regularisation d'inequaltions variationnelles par approximations successives.
              {\it Rev. Francaise D'Inform. Rech. Oper.} {\bf 47} (1970), 154--159.

\bibitem{MH} H. Mine and M. Fukushima, A minimization method for the sum of a convex function and a continuously differentiable function. {\it J. Optim. Theory Appl.} {\bf 33} (1981), 9--23.

\bibitem{Mou} A. Moudafi, Proximal point algorithm extended to equilibrium problems. {\it J. Nat. Geom.} {\bf 15} (1999), 91--100.

\bibitem{MM2}  A. Moudafi and P. E. Maing\'{e}, On the convergence of an approximate proximal method for DC functions. {\it  J. Comput. Math.} {\bf 24} (2006), 475--480.

\bibitem{MB1} B. S. Mordukhovich, {\it Variational Analysis and Generalized Differentiation I: Basic Theory}.
             Springer, Berlin (2006).

\bibitem{N} Y. Nesterov, {\it Introductory Lectures on Convex Optimization: A Basic Course, Applied Optimization},
            Kluwer Academic Publ., Boston, Dordrecht, London (2004).


\bibitem{OR} J. M. Ortega and W. C. Rheinboldt, {\it Iterative Solution of Nonlinear Equations in Several Variables}, Academic Press, New-York (1970).


\bibitem{ParikhBoyd2014} N. Parikh and S. Boyd, {\it Proximal Algorithms}. Foundations and Trends$\textregistered$ in Optimization. {\bf 1} (2014), 127--239.

\bibitem{DC} T. Pham Dinh and H. A. Le Thi, A DC optimization algorithm for solving the trust-region subproblem.
             {\it SIAM J.  Optim.} {\bf 8} (1998), 476--505.

\bibitem{SONS}  S. S. Souza, P. R. Oliveira, J. X. C. Neto, and A. Soubeyran, A proximal method with separable Bregman distances for quasiconvex minimization over the nonnegative orthant. {\it European J. Oper. Res.} {\bf 201}  (2010),  365--376.

\bibitem{SSC} W. Y.  Sun,  R. J. B.  Sampaio, and  M. A. B.  Condido, Proximal point algorithm for minimization of DC function, {\it J. Comput. Math.} {\bf 21} (2003),  451--462.

\bibitem{TA1} P. D. Tao and L. T. H. An, Convex analysis approach to D.C. programming: theory, algorithms and applications, {\it Acta Math. Vietnam}. {\bf 22} (1997), 289--355.

\bibitem{TA2} P. D. Tao and L. T. H. An, Optimisation d.c. (difference de deux fonctions convexes). Dualit$\acute{e}$ et stabilit$\acute{e}$. Optimalit$\acute{e}$s locale et globale. Algorithmes de l'optimisation d.c. (DCA). {\it Technical report, LMI-CNRS URA 1378, INSA-Rouen} (1994).

\bibitem{PS} T. Pham Dinh and  E. B. Souad, {\it Algorithms for solving a class of nonconvex optimization problems: methods of subgradient.} in Fermat Days 85: Mathematics for Optimization, North-Holland Math. Stud. 129, Elsevier, Amsterdam, (1986), 249--270.

\bibitem{RRT} R. T. Rockafellar, {\it Convex Analysis}. Princeton University Press, Princeton, NJ (1970).

\bibitem{ROC} R. T. Rockafellar, Monotone operators and the proximal point algorithm. {\it SIAM J. Control Optim.} {\bf 14} (1976), 877--898.

\bibitem{RWR} R. T. Rockafellar and R. Wets, {\it Variational Analysis}. Grundlehren der Mathematischen Wissenschaften, 317, Springer (1998).

\bibitem{SK} Y. Saeki and D. Kuroiwa, Optimality conditions for DC programming problems with reverse convex constraints. {\it Nonlinear Anal.} {\bf 80} (2013), 18--27.

\bibitem{Tibshirani1996} R. Tibshirani, Regression shrinkage and selection via the lasso. {\it J. R. Stat. Soc. B.} {\bf 58} (1996), 267--288.
\end{thebibliography}
\end{document}